\documentclass{amsart}
\usepackage{amsmath}
\usepackage{paralist}
\usepackage{amsfonts}
\usepackage{amssymb}
\usepackage{amsthm}
\usepackage{amscd}
\usepackage{amsrefs}
\usepackage[colorlinks=true]{hyperref}
\hypersetup{urlcolor=blue, citecolor=red}

\newtheorem{theorem}{Theorem}[section]

\newtheorem{lemma}[theorem]{Lemma}
\newtheorem{proposition}[theorem]{Proposition}

\theoremstyle{definition}
\newtheorem{definition}[theorem]{Definition}
\newtheorem{remark}[theorem]{Remark}

\newcommand{\T}{\mathbb{T}}
\newcommand{\R}{\mathbb{R}}
\newcommand{\C}{\mathbb{C}}
\newcommand{\Z}{\mathbb{Z}}
\newcommand{\N}{\mathbb{N}}

\begin{document}
\title[Periodic modified Benjamin-Ono equation below $H^{1/2}(\mathbb{T})$]{On a priori estimates and existence of periodic solutions to the modified Benjamin-Ono equation below $H^{1/2}(\mathbb{T})$}

\author{Robert Schippa}
\address{Fakult\"at f\"ur Mathematik, Universit\"at Bielefeld, Postfach 10 01 31, 33501 Bielefeld, Germany}
\email{robert.schippa@uni-bielefeld.de}

\begin{abstract}
A priori estimates and existence of real-valued periodic solutions to the modified Benjamin-Ono equation with initial data in $H^s$ for $s>1/4$ are proved locally in time. The approach relies on frequency dependent time localization, after which dispersive properties from Euclidean space are recovered. The same regularity results are proved for the cubic derivative nonlinear Schr\"odinger equation.
\end{abstract}

\maketitle

\bigskip

\section{Introduction}
\label{section:Introduction}
A priori estimates and the existence of periodic solutions to Schr\"odinger-like equations with cubic derivative nonlinearity are discussed.
This includes the modified Benjamin-Ono equation 
\begin{equation}
\label{eq:mBO}
\left\{\begin{array}{cl}
\partial_t u + \mathcal{H} \partial_{xx} u &= \pm \partial_x (u^3)/3 ,  \; (t,x) \in \R \times \T, \\
u(0) &= u_0 \in H_{\R}^s(\mathbb{T}), \end{array} \right.
\end{equation}
where real-valued initial data $u_0$ are considered on the circle $\mathbb{T}=\mathbb{R}/(2 \pi \mathbb{Z})$. Throughout this paper, $\mathcal{H}$ denotes the Hilbert transform, i.e.,
\begin{equation*}
\begin{split}
\mathcal{H}:L^2(\mathbb{T}) &\rightarrow L^2(\mathbb{T}), \\
f &\mapsto \mathcal{F}_x^{-1}(-i \mathrm{sgn}(\cdot) \mathcal{F}_x(f)(\cdot)).
\end{split}
\end{equation*}
Conserved quantities of the flow are the mass
\begin{equation*}
\int_{\T} u^2(x,t) dx = \int_{\T} u_0^2 dx
\end{equation*}
and the energy
\begin{equation*}
E(u) = \frac{1}{2} \int_{\T} (D_x^{1/2} u)^2 dx \mp \int_{\T} \frac{u^4}{12} dx,
\end{equation*}
with the upper/lower sign matching the one from \eqref{eq:mBO}. For the local-in-time analysis in this paper, the focusing properties will not be relevant.

It turns out that the following derivative nonlinear Schr\"odinger equation (dNLS) is also amenable to the employed arguments:
\begin{equation}
\label{eq:dNLS}
\left\{\begin{array}{cl}
i \partial_t u + \partial_{xx} u &= i \partial_x (|u|^2 u) ,  \quad (t,x) \in \R \times \mathbb{T}, \\
u(0) &= u_0 \in H^s(\mathbb{T}). \end{array} \right.
\end{equation}

From the point of view of dispersive equations, the models look very similar. However, \eqref{eq:dNLS} is completely integrable (see \cite{KaupNewell1978}), in contrast to \eqref{eq:mBO}, which is not known to be completely integrable. Since it is useful to point out that the methods of this article do not crucially hinge on complete integrability, we choose to analyze \eqref{eq:mBO} in detail. A discussion of modifications for the analysis of \eqref{eq:dNLS} is postponed to the end of the article.

On the real line, the equations share the scaling symmetry
\begin{equation}
\label{eq:scalingSymmetry}
u(t,x) \rightarrow \lambda^{-1/2} u(\lambda^{-2} t, \lambda^{-1} x),
\end{equation}
which leads to the scaling critical regularity $s_c = 0$, but it is well-known that the data-to-solution mapping fails to be $C^3$ for $s<1/2$. We give a more detailed account on the local well-posedness theory next.

On the real line, \eqref{eq:mBO} was analyzed by Guo in \cite{Guo2011MBO}. It was shown that \eqref{eq:mBO} is locally well-posed for complex-valued initial data with uniform continuity of the data-to-solution mapping provided that $s \geq 1/2$ and that the $L^2$-norm of the initial data is sufficiently small. For smooth and real-valued solutions a priori estimates have been established for $s>1/4$. See also the earlier work \cite{MolinetRibaud2004} and references therein.

For real-valued solutions on the circle local well-posedness in $H^{1/2}(\mathbb{T})$ was shown by Guo et al. in \cite{GuoLinMolinet2014}.

On the real line, Takaoka showed in \cite{Takaoka1999} that \eqref{eq:dNLS} is locally well-posed in $H^{1/2}(\mathbb{R})$ making use of Fourier restriction spaces and a gauge transform to remedy the problematic nonlinear term $|u|^2 \partial_x u$. Global well-posedness in $H^s(\mathbb{R})$, $s>1/2$, with sufficiently small $L^2(\mathbb{R})$-norm was later shown by employing the $I$-method in \cite{CollianderKeelStaffilaniTakaokaTao2002} (see \cite{MiaoWuXu2011} for $s=1/2$). Moreover, in \cite{Gruenrock2005} Gr\"unrock showed local well-posedness in almost scaling critical Fourier Lebesgue spaces.

Adapting the Fourier restriction spaces and the gauge transform to the periodic setting, Herr showed in \cite{Herr2006} that \eqref{eq:dNLS} is locally well-posed in $H^{1/2}(\mathbb{T})$. Again, the data-to-solution mapping fails to be $C^3$ below $H^{1/2}(\T)$ (see \cite{Herr2006}) and even fails to be uniformly continuous below $H^{1/2}(\T)$ (see \cite{Mosincat2018}). Takaoka showed in \cite{Takaoka2016} the existence of weak solutions and a priori estimates for $s>12/25$ conditional upon small $L^2(\mathbb{T})$-norm.

Global well-posedness of \eqref{eq:dNLS} in the periodic setting was shown in $H^s(\mathbb{T})$, $s \geq 1/2$, with sufficiently small $L^2(\mathbb{T})$-norm in \cite{Mosincat2018} (see also \cite{MosincatOh2015}).

Moreover, in \cite{GruenrockHerr2008} Gr\"unrock-Herr proved that \eqref{eq:dNLS} is locally well-posed in Fourier Lebesgue spaces, which scale like $H^{(1/4)+}(\T)$. This was recently claimed to be improved to almost scaling critical Fourier Lebesgue spaces in \cite{DengNahmodYue2019}.

The purpose of this note is to show that the methods from \cite{Guo2011MBO} to show a priori estimates for Sobolev regularities, where the data-to-solution mapping fails to be uniformly continuous, extend to the periodic case.

The key observation is that after localization in time to small frequency dependent time intervals we can recover dispersive properties observable in Euclidean space. This recovers Strichartz estimates on frequency dependent time scales.

Short-time linear Strichartz estimates on general compact manifolds were proved by Burq et al. in \cite{BurqGerardTzvetkov2004}, and bilinear short-time Strichartz estimates on the circle were shown by Moyua-Vega \cite{MoyuaVega2008} (see also \cite{Hani2012} for general compact manifolds). Since the short-time Strichartz estimates on compact manifolds resemble the Euclidean case, we obtain the same regularity for a priori estimates in the periodic setting. None of the aforementioned estimates on compact manifolds can hold true for a time-scale, which does not depend on the size of the frequencies under consideration. The following theorem is proved:
\begin{theorem}
\label{thm:mainResultMBO}
Let $s>1/4$ and $u_0 \in H_{\mathbb{R}}^s(\T)$. There is $T=T(s,\Vert u_0 \Vert_{H^s})$ such that there is a solution $u \in C([-T,T],H^s(\T))$ to \eqref{eq:mBO} in the sense of generalized functions, and we find the a priori estimate
\begin{equation*}
\sup_{t \in [-T,T]} \Vert u(t) \Vert_{H^s} \leq C(s,\Vert u_0 \Vert_{H^s}) \Vert u_0 \Vert_{H^s}
\end{equation*}
to hold. Moreover, we have $C(s,\Vert u_0 \Vert_{H^s}) \leq C_s$ and $T(s,\Vert u_0 \Vert_{H^s}) \geq 1$ as $\Vert u_0 \Vert_{H^s} \rightarrow 0$.
\end{theorem}
The deployed arguments can be perceived as a combination of the perturbative approach and the energy method. We will use adapted function spaces (cf. \cite{HadacHerrKoch2009}) to capture the dispersive effects. To remedy the derivative loss, we localize time reciprocally to the frequency size. This also requires to prove energy estimates.

 This approach was presented by Ionescu et al. in \cite{IonescuKenigTataru2008} in the framework of short-time Fourier restriction spaces, but see also the precursing works of Koch-Tzvetkov \cite{KochTzvetkov2003} and Christ-Colliander-Tao \cite{ChristCollianderTao2008}, where Strichartz estimates on frequency dependent time scales were utilized.

Recall that for a general dispersive equation (see e.g. \cite{Tao2006} for notation)
\begin{equation}
\label{eq:genDispersiveEquation}
\left\{\begin{array}{cl}
i \partial_t u + \omega(\nabla/i) u &= F(u) ,  \; (t,x) \in \mathbb{R} \times M, \; M \in \{ \mathbb{T}, \mathbb{R} \}, \\
u(0) &= u_0 \in H^s(M), \end{array} \right.
\end{equation}
we have the linear energy estimate in Fourier restriction spaces
\begin{equation}
\label{eq:XsbEnergyEstimate}
\Vert \eta(t) u \Vert_{X_{\omega}^{s,b}} \lesssim_{\eta} \Vert u_0 \Vert_{H^s} + \Vert F(u) \Vert_{X_{\omega}^{s,b-1}}
\end{equation}
for $b>1/2$. Hence, one has to prove a nonlinear estimate 
\begin{equation*}
\Vert F(u) \Vert_{X_{\omega}^{s,b-1}} \lesssim G(\Vert u \Vert_{X_{\omega}^{s,b}})
\end{equation*}
to carry out a contraction mapping argument. In this article, we will not use Fourier restriction norms, but adapted norms, which can be viewed as a refinement, and which do not require Fourier transformations in time.

Performing a frequency dependent time localization erases the dependence of the short-time adapted norm on the initial value. This leads to an estimate of the short-time adapted norm $F^{s}(T)$ in terms of a norm $N^s(T)$ for the nonlinearity and an energy norm $E^{s}(T)$. The energy norm takes into account the contribution of dyadic frequency ranges uniform in time before summing (see Lemma \ref{lem:ShorttimeEnergyEstimate}).

Therefore, one also has to propagate this energy norm in terms of the short-time adapted norm. This is carried out in Proposition \ref{prop:shorttimeEnergyEstimateMBO}. As for the usual Fourier restriction norms, one has to estimate the nonlinearity in the $N^s(T)$-norm in terms of the short-time adapted norm (see Proposition \ref{prop:shorttimeTrilinearEstimateMBO}).

For $s >1/4$ and initial data with sufficiently small $H^s$-norm, we show the bounds (cf. \cite{IonescuKenigTataru2008,KochTataru2007}) for solutions $u$ to \eqref{eq:mBO}
\begin{equation}
\label{eq:EstimatesIntroduction}
\left\{\begin{array}{cl}
\Vert u \Vert_{F^s(T)} &\lesssim \Vert u \Vert_{E^s(T)} + \Vert \partial_x(u^3/3) \Vert_{N^s(T)} \\
\Vert \partial_x(u^3/3) \Vert_{N^s(T)} &\lesssim \Vert u \Vert_{F^s(T)}^3 \\
\Vert u \Vert_{E^s(T)}^2 &\lesssim \Vert u_0 \Vert_{H^s}^2 + T \Vert u \Vert_{F^{s-\varepsilon}(T)}^6,
\end{array} \right.
\end{equation}
and the proof will be concluded by a continuity argument. To construct solutions via compactness arguments, a smoothing effect in the energy estimate will be used.

To deal with arbitrary initial data, we rescale to reduce to small initial data on the rescaled torus. However, the $H^s_\lambda$-norm  of a fixed initial value converges to the $L^2(\lambda \T)$-norm upon rescaling $\lambda \rightarrow \infty$ because the equation is $L^2$-critical. Thus, the $L^2$-component of the initial data is insensitive to rescaling. The remedy is to introduce a low frequency weight giving rise to a subcritical norm. This will be detailed in Section \ref{section:Notation}, and we refer to the end of Section \ref{section:Notation} for further discussion. The rescaling argument was previously used for periodic solutions to the modified Korteweg-de Vries equation in the short-time context in \cite{Molinet2012}; the idea to use a subcritical Sobolev norm, which distinguishes high and low frequencies to achieve smallness upon rescaling, was deployed on the real line in \cite{Guo2011MBO}.

An obstacle to control the nonlinear interaction for derivative nonlinearities is $High \times Low \rightarrow High$-interaction as this requires to overcome one whole derivative. In Lemma \ref{lem:HighLowLowHighInteractionMBO} we prove that the time localization $\Delta T= \Delta T(N)=N^{-1}$ allows us to control the derivative loss via a short-time bilinear Strichartz estimate and gives rise to the regularity threshold $s>1/4$ for the short-time nonlinear estimate.

To illustrate the underlying principle, assume that $k_1 \leq k_2 \leq k_3-5$. Let the $P_{k_i}$ denote frequency projectors to frequencies of size comparable to $2^{k_i}$ and $u(t)$, $v(t)$, $w(t)$ free solutions. The recovery of the derivative loss is based on the observation
\begin{equation}
\label{eq:ControlDerivativeLoss}
\begin{split}
&\quad \Vert \partial_x P_{k_4} (P_{k_1} u(t) P_{k_2} v(t) P_{k_3} w(t)) \Vert_{L_t^1([0,T],L_x^2)} \\
&\lesssim T^{1/2} 2^{k_4} \Vert P_{k_1} u(t) P_{k_2} v(t) P_{k_3} w(t) \Vert_{L_t^2([0,T],L_x^2)} \\
&\lesssim T^{1/2} 2^{k_4} \Vert P_{k_1} u(t) \Vert_{L_t^\infty L_x^\infty} \Vert P_{k_2} v(t) P_{k_3} w(t) \Vert_{L_t^2([0,T], L_x^2)} \\
&\lesssim T^{1/2} 2^{k_4} 2^{k_1/2} \Vert u(0) \Vert_{L_x^2} 2^{-k_3/2} \Vert v(0) \Vert_{L_x^2} \Vert w(0) \Vert_{L_x^2} \\
&\lesssim 2^{k_1/2} \Vert u(0) \Vert_{L_x^2} \Vert v(0) \Vert_{L_x^2} \Vert w(0) \Vert_{L_x^2}. 
\end{split}
\end{equation}
The first and second inequality is due to H\"older's inequality, and the third is due to Bernstein's inequality and a short-time bilinear Strichartz estimate, which is only valid for periodic solutions provided that $T \lesssim 2^{-k_4}$.

 It seems possible to improve the short-time nonlinear estimate by increasing time localization. But this would make the energy estimate worse as time intervals $[0,T]$ have to be partitioned into $T (\Delta T(N))^{-1}$ subintervals for frequencies with size comparable to $N$, and $1/4$ seems to be the regularity threshold for the energy estimate with $\Delta T(N)=N^{-1}$. On the other hand, it is not clear how to overcome the derivative loss for decreased frequency dependent time localization. This distinguishes $\Delta T(N)=N^{-1}$.

Furthermore, the energy estimate strongly hinges on the symmetries of solutions. Here, a variant of the $I$-method (cf. \cite{KochTataru2007,CollianderKeelStaffilaniTakaokaTao2003}) is used to introduce correction terms to the Sobolev energies. The underlying symmetrization arguments break down for differences of solutions. Actually, this failure is inevitable due to the well-known breakdown of uniform continuity of the data-to-solution mapping below $H^{1/2}$. Thus, the approach only yields a priori estimates and no continuous dependence on the initial data, as this would require an energy estimate for differences of solutions. We refer to the beginning of Section \ref{section:EnergyEstimatesMBO} for further discussion.

The arguments can be applied to \eqref{eq:dNLS} yielding the same results as in Theorem \ref{thm:mainResultMBO}. This method is divergent from Takaoka's approach \cite{Takaoka2016}, where a gauge transform was applied to ameliorate the derivative loss, and the analysis from \cite{CollianderKeelStaffilaniTakaokaTao2002} and \cite{Herr2006} was combined to prove regularity results below $H^{1/2}(\mathbb{T})$ conditional upon small $L^2$-norm.

The paper is organized as follows: in Section \ref{section:Notation} we introduce notations, and in Section \ref{section:ProofMBO} we show how to conclude the proof of Theorem \ref{thm:mainResultMBO} with the above set of estimates. The proof of the short-time trilinear estimate is carried out in Section \ref{section:ShorttimeTrilinearEstimateMBO} and relies on short-time Strichartz estimates. These are discussed in Section \ref{section:shorttimeLinearEstimates}, and the propagation of the energy norm via short-time Strichartz estimates is carried out in Section \ref{section:EnergyEstimatesMBO}. In Section \ref{section:ModificationsDNLS} we discuss necessary modifications to prove the corresponding regularity results for \eqref{eq:dNLS}.
\section{Notation and Function spaces}
\label{section:Notation}
In this section we collect notations and record basic properties of the utilized short-time function spaces. Most of the properties we consider below were already pointed out in \cite{ChristHolmerTataru2012} for the correspondent spaces on the real line. With the proofs carrying over, we shall be brief.

Since we consider the models with general spatial period $2 \pi \lambda$, we also consider function spaces with general period.
When the subscript $\lambda$ is omitted in the description of a function spaces, the space with $\lambda = 1$ is referred to. The Lebesgue spaces on $\lambda \mathbb{T} := \mathbb{R}/ (2 \pi \lambda \mathbb{Z})$ are defined by
\begin{equation*}
\Vert f \Vert_{L_x^p(\lambda \T)} = \left( \int_0^{2 \pi \lambda} |f(x)|^p dx \right)^{1/p} \quad f:\lambda \T \rightarrow \C,
\end{equation*}
where $p \in [1,\infty]$ with the usual modification for $p = \infty$. We have to keep track of possible dependencies of constants on the spatial scale $\lambda$ and use the same conventions as in \cite{CollianderKeelStaffilaniTakaokaTao2003}. Let $(d\xi)_{\lambda}$ be the normalized counting measure on $\Z/\lambda$:
\begin{equation*}
\int a(\xi) (d\xi)_{\lambda} := \frac{1}{\lambda} \sum_{\xi \in \Z/\lambda} a(\xi).
\end{equation*}
The Fourier transform on $\lambda \T$ is defined by
\begin{equation*}
\hat{f}(\xi) = \int_{\lambda \T} f(x) e^{-i \xi x} dx \quad (\xi \in \Z/\lambda),
\end{equation*}
and the Fourier inversion formula is given by
\begin{equation*}
f(x) = \frac{1}{2 \pi} \int \hat{f}(\xi) e^{ix \xi} (d\xi)_{\lambda}.
\end{equation*}
We find the usual properties of the Fourier transform to hold:
\begin{align*}
\Vert f \Vert_{L_x^2(\lambda \T)} &= \frac{1}{2 \pi} \Vert \hat{f} \Vert_{L^2_{(d\xi)_{\lambda}}} \quad \textrm{(Plancherel)}, \\
\int_{\lambda \T} f(x) \overline{g(x)} dx &= \frac{1}{2 \pi} \int \hat{f}(\xi) \overline{\hat{g}(\xi)} (d\xi)_{\lambda} \quad \textrm{(Parseval)}.
\end{align*}
For further properties, see \cite[p.~702]{CollianderKeelStaffilaniTakaokaTao2003}. For $l,s \geq 0$, we define the Sobolev space $H^{s,l}_{\lambda}$ of $L^2$-functions with finite norm
\begin{equation*}
\Vert f \Vert^2_{H^{s,l}_{\lambda}} = |\hat{f}(0)|^2 + \Vert \hat{f}(\xi)  |\xi|^{l} \Vert^2_{L^2_{(d\xi)_{\lambda}}(|\xi| \leq 1)} + \Vert \hat{f}(\xi) |\xi|^{s} \Vert^2_{L^2_{(d\xi)_{\lambda}}(|\xi| > 1)}
\end{equation*}
and consider the set $H^\infty_{\lambda} = \bigcap_{s} H^{s,0}_\lambda$. To avoid a separate analysis of the zero frequency, we restrict to the subset of functions with vanishing mean. For solutions to \eqref{eq:mBO} or \eqref{eq:dNLS} this means no loss of generality because the mean is a conserved quantity. In the following, we confine ourselves to solutions with vanishing mean.

After rescaling as in \eqref{eq:scalingSymmetry}, \eqref{eq:mBO} becomes
\begin{equation}
\label{eq:RescaledSolution}
\left\{ \begin{array}{cl}
\partial_t u^\lambda + \mathcal{H} \partial_{xx} u^\lambda &= \partial_x ((u^\lambda)^3/3), \quad (t,x) \in \mathbb{R} \times \lambda \mathbb{T} \\
u^\lambda(0) &= u^\lambda_0 \in H^s(\lambda \mathbb{T})
\end{array} \right.
\end{equation}
with $u^\lambda(t,x)=\lambda^{-1/2} u(\lambda^{-2} t, \lambda^{-1} x)$ and $\mathcal{H}$ denotes the Hilbert transform on $\lambda \T$, which is defined as a Fourier multiplier like on $\mathbb{T}$.

For a $2 \pi \lambda$-space-periodic function $v: \mathbb{R} \times \lambda \T \rightarrow \C$ with time variable $t \in \R$, we define the space-time Fourier transform
\begin{equation*}
\tilde{v}(\tau,\xi) = (\mathcal{F}_{t,x} v)(\tau,\xi) = \int_{\R} dt \int_{\lambda \T} dx e^{-ikx} e^{-it \tau} v(t,x) \quad (\xi \in \Z/\lambda, \; t \in \R).
\end{equation*}
The space-time Fourier transform is inverted by
\begin{equation*}
v(t,x) = \frac{1}{(2 \pi)^2} \int_{\R} d\tau \int_{\Z / \lambda } (d\xi)_{\lambda} e^{i x \xi} e^{i t \tau} \tilde{v}(\tau,\xi).
\end{equation*}

Set $\mathbb{N}_0 = \mathbb{N} \cup \{ 0 \}$. We denote dyadic numbers by capital letters $N,K,J \in 2^{\mathbb{Z}}$ and their binary logarithm by the corresponding minuscules $n,k,j \in \mathbb{Z}$.  We consider unions of intervals $I_{n} = \left\{\xi \in \mathbb{R} \; | \; |\xi| \in [N, 2N) \right\}, \;  N = 2^n, \; n \in \mathbb{Z}$. The frequency space is partitioned by $(I_n)_{n \in \Z}$ and $\{ 0 \}$.
 
The Littlewood-Paley projector onto frequencies of size $N=2^n$, $n \in \mathbb{Z}$ is denoted by $P_n:L^2(\lambda \mathbb{T}) \rightarrow L^2(\lambda \mathbb{T})$, that is $(P_n u) \widehat (\xi) = 1_{I_n}(\xi) \hat{u}(\xi)$, and we write for the zero frequency $(P_{[0]} u) \widehat (\xi) = 1_{ \{0 \} }(\xi) \hat{u}(\xi)$.

The dispersion relation \eqref{eq:genDispersiveEquation} for the Benjamin-Ono equation is given by $\omega(\xi)=-\xi |\xi|$. We remark that the properties of the function spaces reviewed in this section are independent of the dispersion relation, and a more general setting is considered in the following.

To maximize the gain in the modulation variable in \eqref{eq:XsbEnergyEstimate}, it is desirable to choose $b=1/2$. However, $H^{1/2}(\R)$ fails to embed into $L^\infty(\mathbb{R})$, and thus, properties of free solutions do not transfer to $X^{s,1/2}_{\omega}$-functions \\
(cf. \cite[Lemma~2.9,~p.~100]{Tao2006}). Our remedy is to work with adapted function spaces, namely $U^2$-spaces, which can be identified as predual space for the space of functions with bounded $2$-variation, $V^2$. In \cite{IonescuKenigTataru2008} a Besov refinement in the modulation was used to cover the endpoint case $b=1/2$. Since an analysis in the modulation variable does not appear to yield further insight, we choose to work with the adapted spaces instead. We contend that this also simplifies some of the proofs of properties for the more classical function spaces introduced in \cite{IonescuKenigTataru2008}.

Here, we collect the most important function space properties for the sake of self-containedness and refer to \cite{HadacHerrKoch2009} (see also \cite{HadacHerrKoch2009Erratum}) for a careful introduction to $U^p$-/$V^p$-spaces.

Let $I=[a,b)$ with $-\infty \leq a < b \leq \infty$. The $V_{\lambda}^p(I)$-spaces contain functions of bounded $p$-variation, $p \in [1,\infty)$, which take values in $L^2(\lambda \T)$ (although the function space properties remain valid for an arbitrary Hilbert space). Also, we indicate $L^p$-spaces for space and time variables by $x$, respectively $t$ in the following. $U_{\lambda}^p(I)$ are atomic spaces, which are predual to the $V_{\lambda}^p(I)$-spaces. We let $\mathcal{Z}(I)$ denote the set of all possible partitions of $I$; these are sequences $a=t_0 < t_1 < \ldots < t_K = b$.
\begin{definition}
Let $\{ t_k \}_{k=0}^K \in \mathcal{Z}(I)$ and $\{ \phi_k \}_{k=0}^{K-1} \subseteq L_x^2(\lambda \T)$ with
\begin{equation*}
\sum_{k=1}^K \Vert \phi_{k-1} \Vert_{L^2_x(\lambda \T)}^p = 1.
\end{equation*}
Then, the function
\begin{equation*}
a(t) = \sum_{k=1}^K \phi_{k-1} \chi_{[t_{k-1},t_k)}(t)
\end{equation*}
is said to be a $U_{\lambda}^p(I)$-atom. Further,
\begin{equation*}
U_{\lambda}^p(I)= \{ f: I \rightarrow L_x^2(\lambda \T) \, | \, \Vert f \Vert_{U_{\lambda}^p(I)} < \infty \},
\end{equation*}
where
\begin{equation*}
\Vert f \Vert_{U_{\lambda}^p(I)} = \inf \{ \Vert \lambda_k \Vert_{\ell_k^1} \, | \, f(t) = \sum_{k=0}^\infty \lambda_k a_k(t), \; a_k - U_{\lambda}^p-\text{atom} \}.
\end{equation*}
\end{definition}

By virtue of the atomic representation, we find elements $u(t) \in U_{\lambda}^p(I)$ to be continuous from the right, having left-limits everywhere and admitting only countably many discontinuities (see \cite[Proposition~2.2,~p.~921]{HadacHerrKoch2009}). For properties of the spaces with bounded $p$-variation, see \cite{Wiener1979}.
\begin{definition}
We set
\begin{equation*}
V_{\lambda}^p(I) = \{ v: I \rightarrow L_x^2(\lambda \T) \, | \, \Vert v \Vert_{V_{\lambda}^p(I)} < \infty \},
\end{equation*}
where
\begin{equation*}
\Vert v \Vert_{V_{\lambda}^p(I)} = \sup_{\{t_k \}_{k=0}^{K-1} \in \mathcal{Z}(I)} \left( \sum_{k=1}^K \Vert v(t_k) - v(t_{k-1}) \Vert_{L^2_x(\lambda \T)}^p \right)^{1/p} < \infty.
\end{equation*}
\end{definition}

We recall that one-sided limits exist for $V_{\lambda}^p$-functions, and $V_{\lambda}^p$-functions can only have countably many discontinuities (see \cite[Proposition~2.4,~p.~922]{HadacHerrKoch2009}).

In the following we confine ourselves to consider the subspaces $V^p_{\lambda,-,rc} \subseteq V_{\lambda}^p$ of right-continuous functions vanishing at $-\infty$. For the sake of brevity, we write $V_{\lambda}^p$ for $V^p_{\lambda,-,rc}$, and occasionally, the subscript $\lambda$ indicating the spatial period will be omitted as the following properties are independent of the base Hilbert space.

\begin{definition}
We define the following subspaces of $V^2$, respectively $U^2$:
\begin{equation*}
\begin{split}
V_0^2(I) &= \{ v \in V^2(I) \, | \, v(a) = 0 \}, \\
U_{0}^2(I) &= \{ u \in U^2(I) \, | \, u(b) = 0 \}.
\end{split}
\end{equation*}
\end{definition}
These function spaces behave well with sharp cutoff functions contrary to $X^{s,b}$-spaces, where one has to use smooth cutoff functions. We have the following estimates for sharp cutoffs (see \cite[Equation~(2.2),~p.~55]{ChristHolmerTataru2012}):
\begin{equation*}
\begin{split}
\Vert u \Vert_{U^p(I)} &= \Vert \chi_I u \Vert_{U^p([-\infty,\infty))}, \\
\Vert v \Vert_{V^p(I)} &\leq \Vert \chi_I v \Vert_{V^p([-\infty,\infty))} \leq 2 \Vert v \Vert_{V^p(I)}.
\end{split}
\end{equation*}
The relation of $U^p$-/$V^p$-spaces with Besov spaces is given by the embeddings (see \cite[Equation~(32),~p.~963]{KochTataru2012})
\begin{equation*}
\dot{B}_1^{1,1/p} \hookrightarrow U^p \hookrightarrow V^p_{rc} \hookrightarrow \dot{B}_{\infty}^{1,1/p}.
\end{equation*}
We record the following further embedding properties:
\begin{lemma}
\label{lem:UpVpFunctionSpaceProperties}
Let $I = [a,b)$.
\begin{enumerate}
\item[1.] If $1 \leq p \leq q < \infty$, then $\Vert u \Vert_{U^q} \leq \Vert u \Vert_{U^p}$ and $\Vert u \Vert_{V^q} \leq \Vert u \Vert_{V^p}$.
\item[2.] If $1 \leq p < \infty$, then $\Vert u \Vert_{V^p} \lesssim \Vert u \Vert_{U^p}$.
\item[3.] If $1 \leq p < q < \infty$, $u(a) = 0$ and $u \in V^p$ is right-continuous, then $\Vert u \Vert_{U^q} \lesssim \Vert u \Vert_{V^p}$.
\item[4.] Let $1 \leq p < q <\infty$, $E$ be a Banach space and $T$ be a linear operator with
\begin{equation*}
\Vert Tu \Vert_E \leq C_q \Vert u \Vert_{U^q}, \; \Vert Tu \Vert_E \leq C_p \Vert u \Vert_{U^p}, \text{ with } 0 < C_p \leq C_q.
\end{equation*}
Then,
\begin{equation*}
\Vert T u \Vert_E \lesssim \log \langle \frac{C_q}{C_p} \rangle \Vert u \Vert_{V^p}.
\end{equation*}
\end{enumerate}
\end{lemma}
\begin{proof}
The first part follows from the embedding properties of the $\ell^p$-norms and the second part from considering $U^p$-atoms. For the third claim, see \cite[Corollary~2.6,~p.~923]{HadacHerrKoch2009}, and the fourth claim is proved in \cite[Proposition~2.20.,~p.~930]{HadacHerrKoch2009}.
\end{proof}

\begin{definition}
We define
\begin{equation*}
DU^2(I) = \{ \partial_t u \, | \, u \in U^2(I) \}
\end{equation*}
with time derivative in the sense of tempered distributions.
\end{definition}
For any $f \in DU^2(I)$, the function $u \in U^2(I)$ satisfying $\partial_t u = f$ is unique up to constants. Fixing the right limit to be zero, we can set
\begin{equation*}
\Vert f \Vert_{DU^2(I)} = \Vert u \Vert_{U^2(I)}, \quad f = \partial_t u, \quad u \in U_0^2,
\end{equation*}
which makes $DU^2(I)$ a Banach space. We have the following embedding property (see \cite[p.~56]{ChristHolmerTataru2012}):
\begin{lemma}
Let $I = [a,b)$. Then,
\begin{equation*}
L^1(I) \hookrightarrow DU^2(I).
\end{equation*}
\end{lemma}
We have the following lemma on $DU-V$-duality:
\begin{lemma}{\cite[Proposition~2.10,~p.~925]{HadacHerrKoch2009}}
We have $(DU^2(I))^* = V_0^2(I)$ with respect to a duality relation, which for $f \in L^1(I) \subseteq DU^2(I)$ is given by
\begin{equation*}
\langle f, v \rangle = \int_a^b \langle f(t), v(t) \rangle_{L_x^2} dt = \int_a^b \int f \overline{v} dx dt.
\end{equation*}
Moreover,
\begin{equation*}
\Vert f \Vert_{DU^2(I)} = \sup_{\Vert v \Vert_{V_0^2} = 1} \left| \int_a^b \int f \overline{v} dx dt \right|.
\end{equation*}
\end{lemma}
For general $f \in DU^2(I)$ one can still consider a related mapping, but this requires more care (cf. \cite[Theorem~2.8,~p.~924]{HadacHerrKoch2009}).

Adapting $U^p$-/$V^p$-spaces to the linear propagator $e^{i t \omega(\nabla/i)}$ yields the following function spaces:
\begin{equation*}
\begin{split}
\Vert u \Vert_{U^p_{\omega}(I;H)} &= \Vert e^{-it \omega(\nabla/i)} u \Vert_{U^p(I;H)}, \\
\Vert v \Vert_{V^p_{\omega}(I;H)} &= \Vert e^{-it \omega(\nabla/i)} v \Vert_{V^p(I;H)}, \\
\Vert u \Vert_{DU^2_{\omega}(I;H)} &= \Vert e^{-it \omega(\nabla/i)} u \Vert_{DU^2(I;H)}.
\end{split}
\end{equation*}
$U^p_\omega$-atoms are piecewise free solutions, which allows us to transfer Strichartz estimates to $U^p$-functions. This will be referred to as transfer principle in the following. For a more precise notion, see \cite[Proposition~2.19,~p.~929]{HadacHerrKoch2009}.
%
%

In this article, we consider the dispersion relations $\omega_{BO}(\xi)=-\xi |\xi|$ and $\omega_{\Delta}(\xi) = - \xi^2$. The subscripts $BO$ and $\Delta$ will indicate the dispersion relation under consideration.

Next, we turn to frequency dependent time localization. Let $T \in (0,1]$ and choose $\Delta T = \Delta T(N) = N^{-1}$ for frequencies with size comparable to $N \in 2^{\mathbb{N}_0}$ following the heuristic \eqref{eq:ControlDerivativeLoss} given in the Introduction. We define the norm of the short-time function space $F_{k,\lambda}$ for functions $u=P_k u$ with $k \in \mathbb{N}_0$ by
\begin{equation*}
\Vert u \Vert^2_{F_{k,\lambda}(T)} = \sup_{\substack{|I|=\min(K^{-1},T), \\ I \subseteq [0,T]}} \Vert P_{k} u \Vert^2_{U^2_{\omega}(I;L^2_x(\lambda \T))} 
\end{equation*}
and similarly,
\begin{equation*}
\Vert f \Vert^2_{N_{k,\lambda}(T)} = \sup_{\substack{|I|=\min(K^{-1},T), \\ I \subseteq [0,T]}} \Vert P_{k} v \Vert^2_{DU^2_{\omega}(I;L^2_x(\lambda \T))}.
\end{equation*}
The function spaces are assembled by Littlewood-Paley decomposition. The short-time space, into which we place the solution, is defined by
\begin{equation*}
\Vert u \Vert^2_{F_{\lambda}^{s,l}(T)} = \Vert P_{[0]} u \Vert^2_{U^2_\omega([0,T],L^2_x(\lambda \T))} + \sum_{k \leq 0} 2^{2kl} \Vert P_k u \Vert^2_{F_{k,\lambda}(T)} + \sum_{k >0 } 2^{2ks} \Vert P_k u \Vert_{F_{k,\lambda}(T)}^2.
\end{equation*}
The function space $N^{s,l}$, into which we will place the nonlinearity, is given by
\begin{equation*}
\begin{split}
\Vert f \Vert^2_{N_{\lambda}^{s,l}(T)} &= \Vert P_{[0]} f \Vert^2_{DU^2_\omega([0,T],L^2_x(\lambda \T))} \\
&\quad + \sum_{k \leq 0} 2^{2kl} \Vert P_k f \Vert_{N_{k,\lambda}(T)}^2 + \sum_{k > 0} 2^{2ks} \Vert P_k u \Vert_{N_{k,\lambda}(T)}^2.
\end{split}
\end{equation*}

The frequency dependent time localization erases the dependence on the initial data away from the origin. Instead of a common energy space $C([0,T],H^{s,l})$, we have to consider the following space: 
\begin{equation*}
\begin{split}
\Vert u \Vert^2_{E_{\lambda}^{s,l}(T)} &= \Vert P_{[0]} u(0) \Vert_{L^2_x(\lambda \T)}^2 + \sum_{k \leq 0} 2^{2kl} \Vert P_k u(0) \Vert^2_{L^2_x(\lambda \T)} \\
&\quad + \sum_{k > 0} 2^{2ks} \sup_{t \in [0,T]} \Vert P_k u(t) \Vert^2_{L^2_x(\lambda \T)}.
\end{split}
\end{equation*}
This space deviates from the usual energy space logarithmically. The following linear estimate substitutes for the energy estimate \eqref{eq:XsbEnergyEstimate}.
\begin{lemma}
\label{lem:ShorttimeEnergyEstimate}
Let $T \in (0,1]$ and $u$ be a solution to \eqref{eq:genDispersiveEquation}. Then, we find the following estimate to hold:
\begin{equation*}
\Vert u \Vert_{F^{s,l}_{\lambda}(T)} \lesssim \Vert u \Vert_{E_{\lambda}^{s,l}(T)} + \Vert F(u) \Vert_{N_{\lambda}^{s,l}(T)}.
\end{equation*}
\end{lemma}
\begin{proof}
A proof in the context of a specific evolution equation, which readily generalizes to arbitrary dispersion relations, is given in \cite[Lemma~3.1.,~p.~59]{ChristHolmerTataru2012}.
\end{proof}
We end the section with a discussion of the use of rescaling, which is often not required for the large data theory. Phrasing the estimates in terms of Fourier restriction spaces, we need the full range of regularity in the modulation variable from $-1/2$ to $1/2$ to prove a nonlinear estimate for $High \times Low \times Low \rightarrow High$-interaction in Lemma \ref{lem:HighLowLowHighInteractionMBO}. Thus, there is no slack in the modulation regularity when it is well-known that modulation regularity can be traded for powers of the time-scale $T$ (cf. \cite[Lemma~2.11,~p.~101]{Tao2006}, or \cite[Lemma~3.4.,~p.~1670]{GuoOh2018} in the short-time context). If this were the case, then we could upgrade the nonlinear estimate from \eqref{eq:EstimatesIntroduction} to
\begin{equation*}
\Vert \partial_x (u^3) \Vert_{N^s(T)} \lesssim T^\theta \Vert u \Vert^3_{F^s(T)}
\end{equation*}
for some $\theta >0$, and the a priori estimates would follow also for large data.

We resort to considering rescaled solutions on $\lambda \mathbb{T}$ with $\lambda \geq 1$, and together with implicit constants in the corresponding set of estimates \eqref{eq:EstimatesIntroduction} independent of $\lambda$, this allows us to prove a priori estimates for initial data with vanishing mean, which are small in the $H^{s,l}_{\lambda}$-norm. Since this is a subcritical norm provided that the considered functions have vanishing mean, this procedure works for arbitrary initial data in $H^{s,l}_{\lambda}$. Alternatively, one could consider weighted norms as in \cite{Zhang2016}; here, we prefer to rescale the torus to illustrate the scale-independence of the argument, which can potentially be used in the analysis of the long-period limit or the global behavior of solutions.
\section{Proof of new regularity results for the mo\-di\-fied Ben\-ja\-min-\-O\-no equation}
\label{section:ProofMBO}
As typical for the construction of solutions, we prove a priori estimates for smooth solutions first. In the second step, we use a compactness argument to prove existence of solutions. For this, we will use a smoothing effect in the energy estimates. In the context of short-time norms, this strategy was previously followed in \cite{GuoOh2018}, where the arguments were given in the context of the cubic nonlinear Schr\"odinger equation. As argued in Section \ref{section:Notation}, it is enough to consider initial data with vanishing mean. This will be implicit in the following. Our first aim is to prove the following proposition:
\begin{proposition}
\label{prop:aPrioriEstimatesSmoothSolutionsMBO}
Let $s >1/4$ and $u_0 \in H_{\mathbb{R}}^{\infty}(\T)$. There is $T=T(s,\Vert u_0 \Vert_{H^s})$ such that we find the estimate
\begin{equation*}
\sup_{t \in [-T,T]} \Vert u(t) \Vert_{H^s} \leq C(s,\Vert u_0 \Vert_{H^s}) \Vert u_0 \Vert_{H^s}
\end{equation*}
to hold for the unique smooth solution $u$ to \eqref{eq:mBO}. Moreover, we find $T \geq 1$ and $C(s,\Vert u_0 \Vert_{H^s}) = D(s)$ as $\Vert u_0 \Vert_{H^s} \rightarrow 0$.
\end{proposition}
The idea is to control the $F^{s,l}_{\lambda}(T)$-norm of the rescaled solution. This suffices to conclude an a priori bound for the Sobolev norm due to $F_{\lambda}^{s,l}(T) \hookrightarrow L^\infty_T H^{s,l}_{\lambda}$.

Continuity and limit properties of $T^\prime \mapsto \Vert u \Vert_{E^{s,l}_{\lambda}(T^\prime)}, \Vert u \Vert_{F^{s,l}_{\lambda}(T)}$ as $T^\prime \rightarrow 0$ to carry out the bootstrap argument are stated in Lemma \ref{lem:energyNormPropertiesMBO}, which was shown in \cite[Section~1]{KochTataru2007}.
\begin{lemma}
\label{lem:energyNormPropertiesMBO}
Suppose that $u \in C([-T,T],H^{\infty}_{\lambda})$ and $u(0) = u_0$. Then, we find the mappings $T^\prime \mapsto \Vert u \Vert_{E^{s,l}_{\lambda}(T^\prime)}$, $T^\prime \mapsto \Vert u \Vert_{F^{s,l}_{\lambda}(T^\prime)}$, $T^\prime \in [0,T)$ to be increasing, continuous, and we have $\limsup_{T^\prime \rightarrow 0} \Vert u \Vert_{E^{s,l}_{\lambda}(T^\prime)} \leq 2 \Vert u_0 \Vert_{H^{s,l}_{\lambda}}$.
\end{lemma}
We are ready to prove Proposition \ref{prop:aPrioriEstimatesSmoothSolutionsMBO}.
\begin{proof}[Proof of Proposition \ref{prop:aPrioriEstimatesSmoothSolutionsMBO}]
In the following we suppose without loss of generality $\int_{\T} u_0 dx = 0$. We start with $\Vert u_0 \Vert_{H^s} \leq \tilde{C}_s \ll 1$. $\tilde{C}_s$ will be specified below, and we shall see how the general case follows from rescaling. Under the smallness assumption on the initial data, by continuity we can invoke Proposition \ref{prop:shorttimeEnergyEstimateMBO} for small times $T^\prime$ and find the following estimates to hold\footnote{Since there are no low frequencies in the present context, the index $l$ is irrelevant and omitted.} from Lemma \ref{lem:ShorttimeEnergyEstimate} and Propositions \ref{prop:shorttimeTrilinearEstimateMBO} and \ref{prop:shorttimeEnergyEstimateMBO}:
\begin{equation*}
\left\{\begin{array}{cl}
\Vert u \Vert_{F^s(T^\prime)} &\leq C_{1,s}  (\Vert u \Vert_{E^s(T^\prime)} + \Vert \partial_x(u^3/3) \Vert_{N^s(T^\prime)}) \\
\Vert \partial_x(u^3/3) \Vert_{N^s(T^\prime)} &\leq C_{2,s} \Vert u \Vert_{F^s(T^\prime)}^3 \\
\Vert u \Vert_{E^s(T^\prime)}^2 &\leq C_{3,s} ( \Vert u_0 \Vert_{H^s}^2 + T^\prime \Vert u \Vert_{F^{s}(T^\prime)}^6 )
\end{array} \right.
\end{equation*}
Following \cite[Section~1]{KochTataru2007}, we set $X(T^\prime) = \Vert u \Vert_{E^s(T^\prime)} + \Vert u \Vert_{F^s(T^\prime)}$ and derive a bound on $X(T^\prime)$ from a continuity argument.\\
Firstly, we find $\limsup_{T^\prime \rightarrow 0} X(T^\prime) \leq 2 \Vert u_0 \Vert_{H^s}$ by Lemma \ref{lem:energyNormPropertiesMBO}. Secondly, we infer from the above estimates that
\begin{equation}
\label{eq:bootstrapBoundMBO}
X(T^\prime) \leq C_s( \Vert u_0 \Vert_{H^s} + X(T^\prime)^3)
\end{equation} 
with $C_s=C_s(C_{1,s},C_{2,s},C_{3,s}) > 1$ for $T^\prime \leq 1$.
From the continuity of $X(T^\prime)$, we have
\begin{equation*}
X(T^\prime) \leq 4C_s \Vert u_0 \Vert_{H^s}
\end{equation*}
for all $T^\prime \in (0,\tilde{T}]$ for some $\tilde{T} \in (0,1]$. However, we find from \eqref{eq:bootstrapBoundMBO} the improvement
\begin{equation*}
X(T^\prime) \leq 2C_s \Vert u_0 \Vert_{H^s}
\end{equation*}
choosing $\tilde{C}_s$ sufficiently small in dependence of $C_s$, e.g. $\tilde{C}_s = (4 C_s)^{-3/2}$.
By a continuity argument, we find
\begin{equation*}
\sup_{t \in [0,1]} \Vert u(t) \Vert_{H^s} \leq 2C_s \Vert u_0 \Vert_{H^s}
\end{equation*}
provided that $\Vert u_0 \Vert_{H^s} \leq \tilde{C}_s$.

Next, we consider the case of initial data large in $H^s$. Fix $0<l<1/4$. We rescale $u_0 \rightarrow \lambda^{-1/2} u_0(\lambda^{-1} \cdot) =: u_0^{\lambda}$ following \eqref{eq:scalingSymmetry}, which also changes the underlying manifold $\T \rightarrow \lambda \T$. For the rescaled initial data, we have $\Vert u_0^\lambda \Vert_{H_{\lambda}^{s,l}} \rightarrow 0$ as $\lambda \rightarrow \infty$.

We have the following set of inequalities for the solutions $u^\lambda$ from \eqref{eq:RescaledSolution}:
\begin{equation*}
\left\{\begin{array}{cl}
\Vert u^{\lambda} \Vert_{F_{\lambda}^{s,l}(T^\prime)} &\leq C_{1,s,l} ( \Vert u^\lambda \Vert_{E_{\lambda}^{s,l}(T^\prime)} + \Vert \partial_x((u^\lambda)^3/3) \Vert_{N_{\lambda}^{s,l}(T^\prime)} \\
\Vert \partial_x((u^\lambda)^3/3) \Vert_{N_{\lambda}^{s,l}(T^\prime)} &\leq C_{2,s,l} \Vert u^\lambda \Vert_{F_{\lambda}^{s,l}(T^\prime)}^3 \\
\Vert u^\lambda \Vert_{E_{\lambda}^{s,l}(T^\prime)}^2 &\leq C_{3,s,l} ( \Vert u^\lambda_0 \Vert_{H_{\lambda}^{s,l}}^2 + T^\prime \Vert u \Vert_{F_{\lambda}^{s,l}(T^\prime)}^6 )
\end{array} \right.
\end{equation*}
By the above means, we find for
\begin{equation*}
X_{\lambda}(T^\prime) := \Vert u^\lambda \Vert_{E^{s,l}_{\lambda}(T^\prime)} + \Vert \partial_x ((u^\lambda)^3/3) \Vert_{N^{s,l}_{\lambda}(T^\prime)} \leq C_{s,l} (\Vert u^\lambda_0 \Vert_{H_{\lambda}^{s,l}}^2 + X_\lambda(T^\prime)^3),
\end{equation*}
which yields
\begin{equation*}
\sup_{t \in [0,1]} \Vert u^\lambda(t) \Vert_{H^{s,l}_{\lambda}} \leq 2 C_{s,l} \Vert u^\lambda_0 \Vert_{H^{s,l}_{\lambda}}
\end{equation*}
provided that $\Vert u_0 \Vert_{H^{s,l}_{\lambda}} \leq \tilde{C}_{s,l} := (4 C_{s,l})^{-3/2}$.
Scaling back, we find the following a priori estimate
\begin{equation*}
\sup_{t \in [0,\lambda^{-2}]} \Vert u(t) \Vert_{H^s} \leq C(s,\Vert u_0 \Vert_{H^s}) \Vert u_0 \Vert_{H^s}
\end{equation*}
to hold.

The dependence of $C$ on $\Vert u_0 \Vert_{H^s}$ comes from an insufficient control over frequencies with size less than $\lambda$ on the unit torus because of the different low frequency weight on the rescaled torus. For these frequencies, we use the estimate due to $L^2$-conservation
\begin{equation*}
\Vert P_{\leq \lambda} u(t) \Vert_{H^s} \lesssim \lambda^s \Vert u_0 \Vert_{L^2} \lesssim \lambda^s \Vert u_0 \Vert_{H^s}.
\end{equation*}
Since we can choose 
\begin{equation*}
\lambda = \left( \frac{\Vert u_0 \Vert_{H^s}}{\tilde{C}_{s,l}} \right)^{1/l},
\end{equation*}
the proof is complete.
\end{proof}
We turn to the proof of existence of solutions. For $u_0 \in H_{\mathbb{R}}^s(\T)$, we denote $u_{0,n} = P_{\leq n} u_0$ for $n \in \N$. Since $u_{0,n} \in H_{\R}^{\infty}(\T)$, there is a sequence of smooth global solutions $(u_n)_{n \in \mathbb{N}}$ to \eqref{eq:mBO} with $u_n(0) = u_{0,n}$, and we have the a priori estimate
\begin{equation*}
\sup_{t \in [0,T_0]} \Vert u_n(t) \Vert_{H^s} \leq C(s,\Vert u_0 \Vert_{H^s}) \Vert u_{0,n} \Vert_{H^s} \leq C(s,\Vert u_0 \Vert_{H^s}) \Vert u_0 \Vert_{H^s}
\end{equation*}
with $T_0$ and $C$ independent of $n$. Next, we prove precompactness of $(u_n)$.
\begin{lemma}
\label{lem:PrecompactnessApproximatingSequenceMBO}
Let $u_0 \in H_{\mathbb{R}}^s(\T)$ for $s>1/4$ and denote by $(u_n)_{n \in \mathbb{N}}$ the sequence of solutions to \eqref{eq:mBO} with $u_n(0) = u_{0,n}$, where $u_{0,n} = P_{\leq n} u_0$. Then, we find the sequence $(u_n)$ to be precompact in $C([-T,T],H^s(\T))$ for $T \leq T_0 = T_0(s,\Vert u_0 \Vert_{H^s})$.
\end{lemma}
\begin{proof}
By the a priori estimate, we have a bound for $\Vert u_n \Vert_{C([-T,T],H^s)}$ uniform in $n$ for $T \leq T_0$. In addition, we prove the following uniform tail estimate as in \cite{GuoOh2018}: for any $\varepsilon > 0$, there is $n_0=n_0(u_0)$ so that we find the estimate
\begin{equation}
\label{eq:uniformTailEstimateMBO}
\Vert P_{\geq n_0} u_n \Vert_{C([-T,T],H^s)} < \varepsilon
\end{equation}
to hold for all $n \in \N$.

This is a consequence of the smoothing effect in the energy estimates from Section \ref{section:EnergyEstimatesMBO}. We consider symbols resembling\footnote{These symbols have to be smoothed out on a scale of size $2^{n_0}$, cf. \cite{GuoOh2018}.}
\begin{equation*}
a(m) = 
\begin{cases}
&\langle m \rangle^{2s}, |m| \geq 2^{n_0} \\
&0, \text{else}
\end{cases}
\end{equation*}
to derive the estimate
\begin{equation*}
\left| \Vert P_{> n_0} u_n \Vert_{E^s(T)}^2 - \Vert P_{> n_0} u_{0,n} \Vert^2_{H^s} \right| \leq C(s,\Vert u_0 \Vert_{H^s}) 2^{-2 \varepsilon n} \rightarrow 0 \text{ as } n_0 \rightarrow \infty.
\end{equation*}
This follows from Proposition \ref{prop:remainderEnergyEstimateMBO}, the embedding $F^{s-\varepsilon}(T) \hookrightarrow L_T^\infty H^{s-\varepsilon}$ and the a priori estimate. Indeed, the choice of symbol implies that there will be two output functions with frequency size at least comparable to $2^{n_0}$. Consequently,
\begin{equation*}
\Vert P_{>n_0} u_n \Vert^2_{C([-T,T],H^s)} \leq \Vert P_{>n_0} u_0 \Vert^2_{H^s} + C(s,\Vert u_0 \Vert_{H^s}) 2^{-2 \varepsilon n_0} \rightarrow 0 \text{ as } n_0 \rightarrow \infty.
\end{equation*}

Hence, it is enough to prove the precompactness of $(P_{\leq n_0} u_n)$ to conclude that of $(u_n)$. From Duhamel's formula and the boundedness of the linear propagator on low frequencies, we find
\begin{equation}
\label{eq:EquicontinuityApproximateSolutionsMBO}
\begin{split}
&\quad \Vert P_{\leq n_0} u_n(t+\delta) - P_{\leq n_0} u_n(t) \Vert_{H^s} \\
&\leq \Vert (e^{i \delta \partial_x^2} - 1) P_{\leq n_0} u_n(t) \Vert_{H^s} + \Vert \int_t^{t+\delta} e^{i(t+\delta-t^\prime) \partial_x^2} P_{\leq n_0} (\partial_x (u_n(t^\prime)^3/3)) dt^\prime \Vert_{H^s} \\
&\lesssim \delta 2^{2n_0} \Vert u_n \Vert_{C([-T,T],H^s)} + 2^{(s+1/2)n_0} \delta \Vert u_n \Vert_{C([-T,T], H^s)}^3 \\
&\lesssim_{n_0,\Vert u_0 \Vert_{H^s}} \delta.
\end{split}
\end{equation}
For the estimate of the first term in the penultimate step, we choose $\delta$ small enough in dependence of $n_0$. For the second term, we use Bernstein's inequality and the Sobolev embedding $H^{1/4} \hookrightarrow L^3$ to write
\begin{equation*}
\begin{split}
\Vert \int_t^{t+\delta} e^{i(t+\delta-t^\prime) \partial_x^2} P_{\leq n_0} (\partial_x (u_n(t^\prime)^3/3)) dt^\prime \Vert_{H^s} &\lesssim 2^{(s+1/2)n_0} \Vert (u_n)^3 \Vert_{L^1([t,t+\delta],L^1)} \\
&\lesssim 2^{(s+1/2)n_0} \delta \Vert u_n \Vert^3_{L_t^\infty L_x^3} \\
&\lesssim 2^{(s+1/2)n_0} \delta \Vert u_n \Vert_{L_t^\infty H^{s}}.
\end{split}
\end{equation*}
The ultimate estimate in \eqref{eq:EquicontinuityApproximateSolutionsMBO} follows also from choosing $\delta$ small enough in dependence of $n_0$ and the a priori estimate.

The equicontinuity of the small frequencies together with the uniform tail estimate \eqref{eq:uniformTailEstimateMBO} implies precompactness by the Arzel\`a-Ascoli criterion. This completes the proof.
\end{proof}
With Proposition \ref{prop:aPrioriEstimatesSmoothSolutionsMBO} and Lemma \ref{lem:PrecompactnessApproximatingSequenceMBO} at disposal, the conclusion of the proof of the main result is an easy consequence of H\"older's inequality and the Sobolev embedding. The details are omitted.

\section{Short-time linear and bilinear Strichartz estimates}
\label{section:shorttimeLinearEstimates}
The building blocks for the short-time nonlinear and energy estimate are linear and bilinear Strichartz estimates. We start with recalling short-time Strichartz estimates for free solutions.

Most of the results are available in the literature for free solutions to the Schr\"o\-din\-ger equation on $\mathbb{T}$. After projecting to negative and positive frequencies and applying the symmetry of motion reversal, we find the estimates also to hold for free periodic solutions to the Benjamin-Ono equation. The case of general period follows from rescaling. Thus, scale invariant estimates are favorable.

Following the heuristic that Schr\"odinger waves localized in frequency with size comparable to $N$ travel with a group velocity proportional to $N$, it is expected that the estimates from Euclidean space remain true on the torus when localized to a time scale of size $N^{-1}$.

Short-time linear Strichartz estimates on compact manifolds were proved in \cite{BurqGerardTzvetkov2004}, which can be stated on $\lambda \T$ as follows:
\begin{proposition}
\label{prop:shorttimeLinearStrichartzEstimate}
Let $\lambda \geq 1$ and $n \in \mathbb{N}_0$. Suppose that $2 \leq q \leq \infty, \; 2 \leq p < \infty$ and $(q,p)$ is Schr\"odinger-admissible, i.e., $\frac{2}{q} + \frac{1}{p} = \frac{1}{2}$ and $u_0 \in L_x^2(\lambda \T)$ with $P_n u_0 = u_0$. Then, we find the following estimate to hold:
\begin{equation}
\label{eq:linearStrichartzEstimate}
\Vert e^{it \partial_x^2} u_0 \Vert_{L_t^q([0,2^{-n}],L_x^p(\lambda \mathbb{T}))} \lesssim_{p,q} \Vert u_0 \Vert_{L^2_x(\lambda \T)}.
\end{equation}
\end{proposition}
\begin{proof}
For $\lambda = 1$ \eqref{eq:linearStrichartzEstimate} is a special case of \cite[Proposition~2.9,~p.~583]{BurqGerardTzvetkov2004}. For general $\lambda$, the claim follows from rescaling.
\end{proof}
This provides us with an epsilon gain in terms of regularity in comparison to the Strichartz estimate for time scales of $\mathcal{O}(1)$:
\begin{equation*}
\Vert e^{it \partial_x^2} u_0 \Vert_{L^6_{t,x}(\T^2)} \leq C_{\varepsilon} N^\varepsilon \Vert u_0 \Vert_{L_x^2(\T)}, \quad \mathrm{supp}(\hat{u}_0) \subseteq [-N,N],
\end{equation*}
which is due to Bourgain (\cite[Proposition~2.36.,~p.~116]{Bourgain1993FourierTransformRestrictionPhenomenaI}).

In Euclidean space, due to the difference in group velocity and global in time dispersive properties, we have the following bilinear Strichartz estimate (cf. \cite[Lemma~111,~p.~270]{Bourgain1998RefinementsStrichartzInequalities} in two dimensions) for $k \leq n-5$:
\begin{equation*}
\Vert P_n e^{it \partial_x^2} u_0 P_k e^{it \partial_x^2} v_0 \Vert_{L_t^2(\R, L_x^2(\R))} \lesssim 2^{-n/2} \Vert P_n u_0 \Vert_{L_x^2(\R)} \Vert P_k v_0 \Vert_{L_x^2(\R)}.
\end{equation*}

After localization in time, we have the following estimate for periodic solutions, which is a special case of \cite[Theorem~1.2.,~p.~343]{Hani2012BilinearOscillatoryIntegralEstimates}:
\begin{proposition}
\label{prop:BilinearPeriodicStrichartzEstimate}
Let $n,k \in \mathbb{N}_0$ with $k \leq n-5$. Suppose that $u_0, \;v_0 \in L_x^2(\lambda \T)$ with $P_n u_0 = u_0$ and $P_k v_0 = v_0$. Then, we find the following estimate to hold:
\begin{equation}
\label{eq:bilinearStrichartzEstimateTorusI}
\Vert e^{it \partial_x^2} u_0 e^{it \partial_x^2} v_0 \Vert_{L_t^2([0,2^{-n}],L^2_x(\lambda \mathbb{T}))} \lesssim 2^{-n/2} \Vert u_0 \Vert_{L^2_x(\lambda \mathbb{T})} \Vert v_0 \Vert_{L^2_x(\lambda \mathbb{T})}.
\end{equation}
\end{proposition}
The estimate
\begin{equation}
\label{eq:bilinearStrichartzEstimateTorusII}
\Vert P_n e^{it \partial_x^2} u_0 \overline{P_k e^{it \partial_x^2} v_0} \Vert_{L_t^2([0,2^{-n}], L^2_x(\lambda \mathbb{T}))} \lesssim 2^{-n/2} \Vert u_0 \Vert_{L^2_x(\lambda \mathbb{T})} \Vert v_0 \Vert_{L^2_x(\lambda \mathbb{T})}
\end{equation}
is also valid and is the rescaled version of \cite[Theorem~2,~p.~120]{MoyuaVega2008}.

In \cite{MoyuaVega2008} is carried out a more precise analysis of bilinear estimates on the torus, also investigating the dependence on the separation of $\text{supp}(\hat{u}_0)$, $\text{supp}(\hat{v}_0)$ and the time-scale. It turns out that it is enough to require 
\begin{equation*}
| |\xi_1| - |\xi_2| | \gtrsim 2^n \text{ for any } \xi_1 \in \text{supp}(\hat{u}_0), \; \xi_2 \in \text{supp}(\hat{v}_0),
\end{equation*}
 and \eqref{eq:bilinearStrichartzEstimateTorusI} and \eqref{eq:bilinearStrichartzEstimateTorusII} remain true. This resembles once more bilinear Strichartz estimates on the real line. We record the following for future use:
\begin{remark}
\label{rem:IntervalSeparation}
Let $I_i \subseteq \mathbb{R}$ be intervals for $i=1,\ldots,4$ and $N,K \in 2^{\mathbb{N}_0}$ with $K \ll N$. Suppose that $\xi_i \in I_i$ satisfy $\xi_1 + \ldots + \xi_4 = 0$ and $|\xi_1| \sim |\xi_2| \sim |\xi_3| \sim N$ and $|\xi_4| \sim K$ whenever $\xi_i \in I_i$, where $i=1,\ldots,4$.

Partition $I_i$ into intervals $(I_i^{(k)})_k$ of length $cN$ for $i=1,2,3$. Then, there is $c$ such that for the intervals $I_i^{(k_i)}$, $i=1,2,3$ with $\xi_i \in I_i^{(k_i)}$ there are $m,n \in \{1,2,3\}$ with $||\xi_m| - |\xi_n|| \gtrsim N$ for any $\xi_m \in I_m^{(k_m)}$ and $\xi_n \in I_n^{(k_n)}$.
\end{remark}
Informally, the observation states that for $\xi_1+\ldots+\xi_4=0$ with $|\xi_1| \sim |\xi_2| \sim |\xi_3| \sim N$, $|\xi_4| \sim K$ and $K \ll N$, there are $\xi_i$ and $\xi_j$ with $i,j \in \{1,2,3\}$ such that $||\xi_i| - |\xi_j|| \gtrsim N$. This will be useful to apply bilinear Strichartz estimates to comparable frequencies.

Furthermore, we record the following refinement for an interaction with very low frequencies involved:
\begin{remark}
\label{rem:BilinearLowFrequencyRefinement}
The bilinear Strichartz estimates remain true when the low frequencies have size smaller than $1$. In case $k \leq 0$ we have the following as a consequence of H\"older's inequality and Bernstein's inequality:
\begin{equation*}
\begin{split}
&\quad \Vert P_n e^{-t \mathcal{H} \partial_{xx}} u_0 P_k e^{-t \mathcal{H} \partial_{xx}} v_0 \Vert_{L_t^2([0,2^{-n}],L^2_x(\lambda \T)} \\
&\lesssim 2^{(k-n)/2} \Vert P_n u_0 \Vert_{L^2_x(\lambda \T)} \Vert P_k v_0 \Vert_{L^2_x(\lambda \mathbb{T})}
\end{split}
\end{equation*}
\end{remark}

Next, the estimates are transferred to short-time function spaces. We start with the Benjamin-Ono case.
\begin{proposition}
\label{prop:StrichartzEstimatesMBO}
Let $m,n \in \mathbb{N}_0$, $m \leq n$ and $u \in U^2_{BO}(I;L^2_x(\lambda \mathbb{T}))$ with $P_n u = u$ and $|I| \lesssim 2^{-n}$. Then, we find the following estimate to hold:
\begin{equation}
\label{eq:LinearStrichartzEstimateMBO}
\Vert u \Vert_{L_t^6(I;L_x^6(\lambda \mathbb{T}))} \lesssim \Vert u \Vert_{U^2_{BO}(I;L^2_x(\lambda \mathbb{T}))}.
\end{equation}

Further, let $v \in U^2_{BO}(I;L^2_x(\lambda \T))$ with $P_m v = v$ such that $2^m \ll 2^n$ or $||\xi_1| - |\xi_2|| \gtrsim 2^n$ whenever $\xi_1 \in \mathrm{supp}_{\xi} (\hat{u})$ and $\xi_2 \in \mathrm{supp}_{\xi} (\hat{v})$. Then, we find the following estimates to hold:
\begin{align}
\label{eq:BilinearStrichartzEstimateMBOU2}
\Vert u v \Vert_{L_t^2(I;L^2_x(\lambda \T))} &\lesssim 2^{-n/2} \Vert u \Vert_{U^2_{BO}(I;L^2_x(\lambda \T))} \Vert v \Vert_{U^2_{BO}(I;L^2_x(\lambda \T))}, \\
\label{eq:BilinearStrichartzEstimateMBOV2}
\Vert u v \Vert_{L_t^2(I;L^2_x(\lambda \T))} &\lesssim n^2 2^{-n/2} \Vert u \Vert_{V^2_{BO}(I;L^2_x(\lambda \T))} \Vert v \Vert_{V^2_{BO}(I;L^2_x(\lambda \T))}.
\end{align}
For the latter estimate, it is enough to assume $u,v \in V^2_{BO}(I;L^2_x(\lambda \T))$.
\end{proposition}
\begin{proof}
It is enough to verify the claims for $\lambda=1$ as the general case follows from rescaling. \eqref{eq:LinearStrichartzEstimateMBO} is a consequence of Proposition \ref{prop:shorttimeLinearStrichartzEstimate} and the transfer principle after considering positive and negative frequencies separately.

\eqref{eq:BilinearStrichartzEstimateMBOU2} follows from the transfer principle and the short-time bilinear Strichartz estimate from Proposition \ref{prop:BilinearPeriodicStrichartzEstimate}, see also the subsequent remark.

\eqref{eq:BilinearStrichartzEstimateMBOV2} follows from interpolating \eqref{eq:BilinearStrichartzEstimateMBOU2} with linear estimates, see Property (iv) from Lemma \ref{lem:UpVpFunctionSpaceProperties}.
\end{proof}
We record the corresponding estimates in case of Schr\"odinger interaction, which follow like in the previous proposition:
\begin{proposition}
\label{prop:StrichartzEstimatesSEQ}
Let $m,n \in \mathbb{N}_0$ and $m \leq n$. Suppose that $u \in U^2_{\Delta}(I;L^2_x(\lambda \T))$ with $P_n u = u$ and $|I| \lesssim 2^{-n}$. Then, we find the following estimate to hold:
\begin{equation*}
\Vert u \Vert_{L_t^6(I;L_x^6(\lambda \T))} \lesssim \Vert u \Vert_{U^2_{\Delta}(I;L^2_x(\lambda \T))}.
\end{equation*}

Further, let $v \in U^2_{\Delta}(I;L^2_x(\lambda \T))$ with $P_m v = v$ such that $2^m \ll 2^n$ or $||\xi_1| - |\xi_2|| \gtrsim 2^n$ whenever $\xi_1 \in \mathrm{supp}_{\xi} (\hat{u})$ and $\xi_2 \in \mathrm{supp}_{\xi} (\hat{v})$. Then, we find the following estimates to hold:
\begin{align}
\label{eq:BilinearStrichartzEstimateSEQU2}
\Vert u v \Vert_{L_t^2(I;L^2_x(\lambda \T))} &\lesssim 2^{-n/2} \Vert u \Vert_{U^2_{\Delta}(I;L^2_x(\lambda \T))} \Vert v \Vert_{U^2_{\Delta}(I;L^2_x(\lambda \T))}, \\
\label{eq:BilinearStrichartzEstimateSEQV2}
\Vert u v \Vert_{L_t^2(I;L^2_x(\lambda \T))} &\lesssim n^2 2^{-n/2} \Vert u \Vert_{V^2_{\Delta}(I;L^2_x(\lambda \T))} \Vert v \Vert_{V^2_{\Delta}(I;L^2_x(\lambda \T))}.
\end{align}

The estimates remain true after replacing $v$ by $\overline{v}$. For the latter estimate, it is enough to assume $u,v \in V^2_{\Delta}(I;L^2_x(\lambda \T))$.
\end{proposition}
Remark \ref{rem:IntervalSeparation} implies that in case of a $High \times High \times High \times Low$-interaction the functions are still amenable to two bilinear estimates after partitioning the Fourier support into finitely many subintervals.
\section{Short-time trilinear estimates}
\label{section:ShorttimeTrilinearEstimateMBO}
The aim of this section is to derive a short-time trilinear estimate for the nonlinear interaction in \eqref{eq:mBO}. In this section the short-time function spaces are adapted to the dispersion relation $\omega(\xi) = -\xi |\xi|$.
\begin{proposition}
\label{prop:shorttimeTrilinearEstimateMBO}
Let $\lambda \geq 1$ and suppose that $0<l<1/4<s$, $T \in (0,1]$ and $u,v,w \in F_{\lambda}^{s,l}(T)$. Then, we find the following estimate to hold:
\begin{equation}
\label{eq:ShorttimeTrilinearEstimateMBO}
\Vert \partial_x (u v w) \Vert_{N^{s,l}_{\lambda}(T)} \lesssim \Vert u \Vert_{F^{s,l}_{\lambda}(T)} \Vert v \Vert_{F^{s,l}_{\lambda}(T)} \Vert w \Vert_{F^{s,l}_{\lambda}(T)}.
\end{equation}
\end{proposition}
We perform decompositions with respect to frequency, essentially reducing the estimate \eqref{eq:ShorttimeTrilinearEstimateMBO} from above to 
\begin{equation}
\label{eq:frequencyLocalizedShorttimeTrilinearEstimateMBO}
\Vert P_{k_4} \partial_x(u_{k_1} v_{k_2} w_{k_3}) \Vert_{N_{k_4,\lambda}} \lesssim \underbrace{\alpha(k_1,k_2,k_3,k_4)}_{\alpha(\underline{k})} \Vert u_{k_1} \Vert_{F_{k_1,\lambda}} \Vert v_{k_2} \Vert_{F_{k_2,\lambda}} \Vert w_{k_3} \Vert_{F_{k_3,\lambda}}.
\end{equation}
For the remainder of this section, the $k_i, \; i=1,\ldots,4$ denote integers and we assume $\text{supp}_{\xi}(\hat{u}_{k_1}) \subseteq I_{k_1}$, $\text{supp}_{\xi}(\hat{v}_{k_2}) \subseteq I_{k_2}$ and similarly for $w_{k_3}$.

We prove \eqref{eq:frequencyLocalizedShorttimeTrilinearEstimateMBO} using the estimates from Proposition \ref{prop:StrichartzEstimatesMBO}. To structure the proof, we list each possible frequency interaction. In any case, we find estimate \eqref{eq:ShorttimeTrilinearEstimateMBO} to hold for regularities $0<l<1/4<s$.
\begin{enumerate}
\item[(i)] $High \times Low \times Low \rightarrow High$-interaction: This interaction will be estimated in Lemma \ref{lem:HighLowLowHighInteractionMBO}.
\item[(ii)] $High \times High \times Low \rightarrow High$-interaction: This interaction will be estimated in Lemma \ref{lem:HighHighLowHighInteractionMBO}.
\item[(iii)] $High \times High \times High \rightarrow High$-interaction: This interaction will be estimated in Lemma \ref{lem:HighHighHighHighInteractionMBO}.
\item[(iv)] $High \times High \times Low \rightarrow Low$-interaction: This interaction will be estimated in Lemma \ref{lem:HighHighLowLowInteractionMBO}.
\item[(v)] $High \times High \times High \rightarrow Low$-interaction: This interaction will be estimated in Lemma \ref{lem:HighHighHighLowInteractionMBO}.
\item[(vi)] $Low \times Low \times Low \rightarrow Low$-interaction: This interaction will be estimated in Lemma \ref{lem:LowLowLowLowInteractionMBO}.
\end{enumerate}
We start with $High \times Low \times Low \rightarrow High $-interaction. The below computation makes the heuristic argument \eqref{eq:ControlDerivativeLoss} from the Introduction precise.
\begin{lemma}
\label{lem:HighLowLowHighInteractionMBO}
Suppose that $k_4 \geq 20$ and $k_1 \leq k_2 \leq k_3 -5$. Then, we find \eqref{eq:frequencyLocalizedShorttimeTrilinearEstimateMBO} to hold with $\alpha(\underline{k}) =2^{k_1/2} \min(1,2^{k_2/2})$.
\end{lemma}
\begin{proof}
We use the embedding $L^1(I) \hookrightarrow DU^2_{BO}(I)$ and H\"older in time to find for $|I| \lesssim 2^{-k_4}$
\begin{equation*}
\begin{split}
&\quad \Vert P_{k_4} \partial_x (u_{k_1} v_{k_2} w_{k_3}) \Vert_{DU^2_{BO}(I;L^2_x(\lambda \T))} \\
&\lesssim \Vert \partial_x (u_{k_1} v_{k_2} w_{k_3}) \Vert_{L_t^1(I;L^2_x(\lambda \T))} \\
&\lesssim 2^{k_4/2} \Vert u_{k_1} v_{k_2} w_{k_3} \Vert_{L_t^2(I;L^2_x(\lambda \T))} \\
&\lesssim 2^{k_4/2} \Vert u_{k_1} \Vert_{L_t^\infty(I;L_x^\infty(\lambda \T))} \Vert v_{k_2} w_{k_3} \Vert_{L_t^2(I;L^2_x(\lambda \T))} \\
&\lesssim 2^{k_1/2} \min(1,2^{k_2/2}) \Vert u_{k_1} \Vert_{U^2_{BO}(I;L^2_x(\lambda \T))} \Vert v_{k_2} \Vert_{U^2_{BO}(I;L^2_x(\lambda \T))} \Vert w_{k_3} \Vert_{U^2_{BO}(I;L^2_x( \lambda \T))}.
\end{split}
\end{equation*}

The ultimate estimate follows from \eqref{eq:BilinearStrichartzEstimateMBOU2}, Bernstein's inequality and Remark \ref{rem:BilinearLowFrequencyRefinement}. The claim follows from the definition of the function spaces.
\end{proof}
Next, $High \times High \times Low \rightarrow High$-interaction is considered:
\begin{lemma}
\label{lem:HighHighLowHighInteractionMBO}
Suppose that $k_4 \geq 20$, $ k_1 \leq k_2 \leq k_3$ and $k_1 \leq k_2-20$, $|k_2-k_4| \leq 10$. Then, \eqref{eq:frequencyLocalizedShorttimeTrilinearEstimateMBO} holds with $\alpha(\underline{k}) = 2^{(0+)k_4} \min(1,2^{k_1/2})$.
\end{lemma}
\begin{proof}
Let $I$ be an interval with $|I| \lesssim 2^{-k_4}$. We use duality to write
\begin{equation}
\label{eq:DualityHighHighLowHighInteractionMBO}
\Vert P_{k_4} \partial_x (u_{k_1} v_{k_2} w_{k_3}) \Vert_{DU^2_{BO}(I;L^2_x(\lambda \T))} \lesssim 2^{k_4} \sup_{\Vert f \Vert_{V^2_{BO}}=1} \int_I \int_{\lambda \T} P_{k_4} f u_{k_1} v_{k_2} w_{k_3} dx dt.
\end{equation}
Among $P_{k_4} f$, $v_{k_2}$, $w_{k_3}$, there is a pair amenable to a bilinear Strichartz estimate following Remark \ref{rem:IntervalSeparation}. Say this is $P_{k_4} f w_{k_3}$. Then, we find from \eqref{eq:BilinearStrichartzEstimateMBOU2} and \eqref{eq:BilinearStrichartzEstimateMBOV2} the following
\begin{equation*}
\begin{split}
\eqref{eq:DualityHighHighLowHighInteractionMBO} &\lesssim 2^{k_4} \Vert P_{k_4} f w_{k_3} \Vert_{L_t^2(I;L^2_x(\lambda \T))} \Vert u_{k_1} v_{k_2} \Vert_{L_t^2(I;L^2_x(\lambda \T))} \\
&\lesssim k_4^2 \min(1,2^{k_1/2}) \left( \sup_{\Vert f \Vert_{V^2_{BO}} = 1} \Vert P_{k_4} f \Vert_{V^2_{BO}} \right) \Vert w_{k_3} \Vert_{V^2_{BO}(I;L^2_x(\lambda \T))} \\
&\quad \quad \times \Vert u_{k_1} \Vert_{U^2_{BO}(I;L^2_x(\lambda \T))} \Vert v_{k_2} \Vert_{U^2_{BO}(I;L^2_x(\lambda \T))} \\
&\lesssim 2^{(0+)k_4} \min(1,2^{k_1/2}) \Vert u_{k_1} \Vert_{U^2_{BO}(I;L^2_x(\lambda \T))} \\
&\quad \quad \times \Vert v_{k_2} \Vert_{U^2_{BO}(I;L^2_x(\lambda \T))} \Vert w_{k_3} \Vert_{U^2_{BO}(I;L^2_x(\lambda \T))},
\end{split}
\end{equation*}
where the ultimate step follows from the embedding properties of $U^p$-/$V^p$-spaces. The claim follows from the definition of the function spaces.
\end{proof}
We estimate the interaction, which leads to the $s=1/2$-threshold of local well-posedness with uniformly continuous dependence on initial data, namely $High \times High \times High \rightarrow High$-interaction:
\begin{lemma}
\label{lem:HighHighHighHighInteractionMBO}
Suppose that $k_4 \geq 20$, $k_1 \leq k_2 \leq k_3$ and $|k_i - k_4| \leq 30$ for any $i \in \{1,2,3\}$. Then, \eqref{eq:frequencyLocalizedShorttimeTrilinearEstimateMBO} holds with $\alpha(\underline{k}) = 2^{k_4/2}$.
\end{lemma}
\begin{proof}
We use the embedding $L^1(I) \hookrightarrow DU_{BO}^2(I)$, H\"older in time and \eqref{eq:LinearStrichartzEstimateMBO} to find for $|I| \lesssim 2^{-k_4}$
\begin{equation*}
\begin{split}
&\quad \Vert P_{k_4} \partial_x (u_{k_1} v_{k_2} w_{k_3}) \Vert_{DU^2_{BO}(I;L^2_x(\lambda \T))} \\
&\lesssim 2^{k_4} \Vert u_{k_1} v_{k_2} w_{k_3} \Vert_{L_t^1(I;L^2_x(\lambda \T))} \\
&\lesssim 2^{k_4/2} \Vert u_{k_1} v_{k_2} w_{k_3} \Vert_{L_t^2(I;L^2_x(\lambda \T))} \\
&\lesssim 2^{k_4/2} \Vert u_{k_1} \Vert_{L_t^6(I;L_x^6(\lambda \T))} \Vert v_{k_2} \Vert_{L_t^6(I;L_x^6(\lambda \T))} \Vert w_{k_3} \Vert_{L_t^6(I;L_x^6(\lambda \T))} \\
&\lesssim 2^{k_4/2} \Vert u_{k_1} \Vert_{U^2_{BO}(I;L^2_x(\lambda \T))} \Vert v_{k_2} \Vert_{U^2_{BO}(I;L^2_x(\lambda \T))} \Vert w_{k_3} \Vert_{U^2_{BO}(I;L^2_x(\lambda \T))},
\end{split}
\end{equation*}
which yields the claim.
\end{proof}
In the following interactions one input frequency is significantly larger than the output frequency. This requires to add localization in time for an estimate in short-time spaces. We consider the contribution from $High \times High \times Low \rightarrow Low$-interaction.
\begin{lemma}
\label{lem:HighHighLowLowInteractionMBO}
Suppose that $k_3 \geq 20$, $k_1 \leq k_2 \leq k_3$, $k_1 \leq k_2 - 10$ and $k_4 \leq k_2-10$. Then, \eqref{eq:frequencyLocalizedShorttimeTrilinearEstimateMBO} holds with $\alpha(\underline{k}) = 2^{(0+)k_3} \min(1,2^{k_1/2}) \min(1,2^{k_4/2})$.
\end{lemma}
\begin{proof}
We use duality to write for $|I| \lesssim 2^{-k_4}$
\begin{equation}
\label{eq:DualityHighHighLowLowInteractionMBO}
\Vert P_{k_4} \partial_x (u_{k_1} v_{k_2} w_{k_3}) \Vert_{DU^2_{BO}(I;L^2_x(\lambda \T))} \lesssim 2^{k_4} \sup_{\Vert f \Vert_{V^2_{BO}} = 1} \int_I \int_{\lambda \T} P_{k_4} f u_{k_1} v_{k_2} w_{k_3} dx dt .
\end{equation}
To estimate $u_{k_1}$ and $v_{k_2}$ in short-time spaces, we have to divide up $I$ into intervals $J$ with $|J| \lesssim 2^{-k_3}$ and write
\begin{equation*}
\begin{split}
\eqref{eq:DualityHighHighLowLowInteractionMBO} &\lesssim \sum_{\substack{J \subseteq I,\\ |J| \lesssim 2^{-k_3}}} 2^{k_4} \sup_{\Vert f \Vert_{V^2_{BO}} = 1} \int_J \int_{\lambda \T} P_{k_4} f u_{k_1} v_{k_2} w_{k_3} dx dt \\
&\lesssim 2^{k_4} \sum_{J \subseteq I} \sup_{\Vert f \Vert_{V^2_{BO}} = 1} \Vert P_{k_4} f v_{k_2} \Vert_{L^2_t(J;L^2_x(\lambda \T))} \Vert u_{k_1} w_{k_3} \Vert_{L^2_t(J;L^2_x(\lambda \T))} \\
&\lesssim 2^{k_4} \min(1,2^{k_1/2}) \min(1,2^{k_4/2}) \sum_{J \subseteq I} k_3^2 2^{-k_3/2} \Vert v_{k_2} \Vert_{U^2_{BO}(J;L^2_x(\lambda \T))} \\
&\quad \quad \times 2^{-k_3/2} \Vert u_{k_1} \Vert_{U^2_{BO}(J;L_x^2(\lambda \T))} \Vert w_{k_3} \Vert_{U^2_{BO}(J;L_x^2(\lambda \T))} \\
&\lesssim 2^{(0+)k_3} \min(1,2^{k_1/2}) \min(1,2^{k_4/2}) \\
&\quad \quad \times \sup_{\substack{J \subseteq I,\\ |J| \lesssim 2^{-k_3}}} \left( \Vert u_{k_1} \Vert_{U^2_{BO}(J;L_x^2(\lambda \T))} \Vert v_{k_2} \Vert_{U^2_{BO}(J;L_x^2(\lambda \T))} \Vert w_{k_3} \Vert_{U^2_{BO}(J;L_x^2(\lambda \T))} \right),
\end{split}
\end{equation*}
where the penultimate estimate follows from \eqref{eq:BilinearStrichartzEstimateMBOU2},  \eqref{eq:BilinearStrichartzEstimateMBOV2} and Remark \ref{rem:BilinearLowFrequencyRefinement}. The ultimate estimate follows from partitioning $I$ with intervals of length $2^{-k_3}$ giving a factor of $2^{k_3-k_4}$ in case $k_4 \geq 0$. For $k_4<0$ the estimate improves. The claim follows from the definition of the function spaces.
\end{proof}
Next, we deal with $High \times High \times High \rightarrow Low$-interaction:
\begin{lemma}
\label{lem:HighHighHighLowInteractionMBO}
Suppose that $k_3 \geq 20$, $k_1 \leq k_2 \leq k_3$, $k_4 \leq k_1-10$ and $k_3-k_1 \leq 10$. Then, we find \eqref{eq:frequencyLocalizedShorttimeTrilinearEstimateMBO} to hold with $\alpha(\underline{k}) = 2^{(0+)k_1} \min(1,2^{k_4/2})$.
\end{lemma}
\begin{proof}
By the above argument, we write
\begin{equation}
\label{eq:DualityHighHighHighLowInteractionMBO}
\begin{split}
&\quad \Vert P_{k_4} \partial_x (u_{k_1} v_{k_2} w_{k_3}) \Vert_{DU^2_{BO}(I;L^2(\lambda \T))} \\
&\lesssim 2^{k_4} \sum_{\substack{J \subseteq I,\\ |J| \lesssim 2^{-k_4}}} \sup_{\Vert f \Vert_{V^2_{BO}} = 1} \int_J \int_{\lambda \T} P_{k_4} f u_{k_1} v_{k_2} w_{k_3} dx dt.
\end{split}
\end{equation}

Now we observe following Remark \ref{rem:IntervalSeparation} that among the high frequencies there must be one pair say $u_{k_1}$, $v_{k_2}$ with $\left| |\xi_1| - |\xi_2| \right| \gtrsim 2^{k_1}$ for $\xi_1 \in \text{supp}_{\xi} ( \hat{u}_{k_1})$ and $\xi_2 \in \text{supp}_{\xi} (\hat{v}_{k_2})$ after partitioning the support of the input functions into finitely many subintervals.

To this pair, we can apply a bilinear Strichartz estimate from Proposition \ref{prop:StrichartzEstimatesMBO} to find
\begin{equation*}
\begin{split}
\eqref{eq:DualityHighHighHighLowInteractionMBO} &\lesssim 2^{k_4} \sum_{\substack{J \subseteq I, \\ |J| = 2^{-k_1}}} \sup_{\Vert f \Vert_{V^2_{BO}} = 1} \Vert P_{k_4} f w_{k_3} \Vert_{L_t^2(J;L_x^2(\lambda \T))} \Vert u_{k_1} v_{k_2} \Vert_{L_t^2(J;L_x^2(\lambda \T))} \\
&\lesssim 2^{k_4} \min(1,2^{k_4/2}) \sum_{J \subseteq I} k_1^2 2^{-k_1/2} \Vert w_{k_3} \Vert_{U_{BO}^2(J;L_x^2(\lambda \T))} \\
&\quad \quad \times 2^{-k_1/2} \Vert u_{k_1} \Vert_{U^2_{BO}(J;L_x^2(\lambda \T))} \Vert v_{k_2} \Vert_{U^2_{BO}(J;L_x^2(\lambda \T))} \\
&\lesssim 2^{(0+)k_1} \min(1,2^{k_4/2}) \\
&\quad \quad \times \sup_{J \subseteq I} \left( \Vert u_{k_1} \Vert_{U^2_{BO}(J;L_x^2(\lambda \T))} \Vert v_{k_2} \Vert_{U^2_{BO}(J;L_x^2(\lambda \T))} \Vert w_{k_3} \Vert_{U^2_{BO}(J;L_x^2(\lambda \T))} \right).
\end{split}
\end{equation*}
The claim follows from the definition of the function spaces.
\end{proof}
At last, we record the $Low \times Low \times Low \rightarrow Low$-estimate, which is straight-forward from H\"older's and Bernstein's inequality.
\begin{lemma}
\label{lem:LowLowLowLowInteractionMBO}
Suppose that $\max_{i=1,2,3,4} k_i \leq 200$ and $k_1 \leq k_2 \leq k_3$. Then, we find estimate \eqref{eq:frequencyLocalizedShorttimeTrilinearEstimateMBO} to hold with $\alpha(\underline{k})=2^{k_1/2} 2^{k_2/2} 2^{k_4}$.
\end{lemma}
\section{Energy estimates}
\label{section:EnergyEstimatesMBO}
In the following the energy norm is propagated in terms of the short-time $U^2$-norms. We shall show the estimate
\begin{equation}
\label{eq:shorttimeEnergyEstimateMBO}
\Vert u \Vert_{E_{\lambda}^{s,l}(T)}^2 \lesssim_{s,l} \Vert u_0 \Vert_{H_{\lambda}^{s,l}}^2 + T \Vert u \Vert_{F_{\lambda}^{s-\tilde{\varepsilon},l}(T)}^6
\end{equation}
for $0<l<1/4<s$, small enough $\sup_{t \in [-T,T]} \Vert u(t) \Vert_{H^{s,l}_{\lambda}}$ and $\tilde{\varepsilon} = \tilde{\varepsilon}(s)>0$. The estimate will be independent of $\lambda$. A similar estimate was proved on the real line in \cite[Proposition~8.1.,~p.~1124]{Guo2011MBO}.
\begin{proposition}
\label{prop:shorttimeEnergyEstimateMBO}
Let $\lambda \geq 1$, $T \in (0,1]$ and $u \in C([-T,T],H_{\lambda}^{\infty})$ be a real-valued solution to \eqref{eq:RescaledSolution}. Then, for $0<l<1/4<s$, there exists $\tilde{\varepsilon}(s)>0$ and $\delta(s)>0$ such that we find \eqref{eq:shorttimeEnergyEstimateMBO} to hold provided that 
\begin{equation*}
\sup_{t \in [-T,T]} \Vert u(t) \Vert_{H^{s,l}_{\lambda}} \leq \delta(s).
\end{equation*}
\end{proposition}
To prove Proposition \ref{prop:shorttimeEnergyEstimateMBO}, we employ a variant of the $I$-method (cf. \cite{CollianderKeelStaffilaniTakaokaTao2002,CollianderKeelStaffilaniTakaokaTao2003}).\\
Symmetrized energy quantities are considered, which come into play after integration by parts in the time variable. In the context of short-time norms, this strategy was previously followed in \cite{KochTataru2007,KochTataru2012}.

The following analysis is close to the arguments on the real line from \cite{Guo2011MBO}. In fact, we see from the proof that one can treat the Euclidean and periodic case simultaneously. However, we prefer to use multilinear estimates over linear estimates as was done in \cite{Guo2011MBO}.

We also make use of the following definition from \cite{KochTataru2007}:
\begin{definition}
\label{def:FrequencyEnvelopeSymbols}
Let $\varepsilon > 0$ and $s \in \mathbb{R}$. Then $S^s_{\varepsilon}$ is the set of positively real-valued, symmetric and smooth functions on the real line (symbols) with the following properties:
\begin{enumerate}
\item[(i)] Slowly varying and support condition: There is $\xi_0 \in 2^{\mathbb{N}_0}$ such that for $\xi \sim \xi^\prime \gtrsim \xi_0$ we have
\begin{align*}
a(\xi) &\sim a(\xi^\prime), \\
a(\xi) &= 0 \text{ for } |\xi| \leq \xi_0/2,
\end{align*}
\item[(ii)] symbol regularity,
\begin{equation*}
\forall \alpha \in \mathbb{N}_0: |\partial^{\alpha} a(\xi)| \lesssim_{\alpha} a(\xi) (1+\xi^2)^{-\alpha/2},
\end{equation*}
\item[(iii)] growth at infinity, for $|\xi| \gg \xi_0$ we have
\begin{equation*}
s-\varepsilon \leq \frac{\log a(\xi)}{\log(1+\xi^2)} \leq s+ \varepsilon.
\end{equation*}
\end{enumerate}
\end{definition}
Note that since $a$ and expressions involving $a$ act as Fourier multiplier for $2 \pi\lambda$-periodic functions, the actually relevant domain of $a$ is $\mathbb{Z}/\lambda$. To derive favorable pointwise estimates, extended versions to the real line are used. Furthermore, if we only wanted to control the $H_\lambda^s$-norm of $u$, then we would only have to take into account the symbols $a(\xi)=(1+\xi^2)^s$, and by the support condition, we emphasize that only high frequencies are analyzed as the estimate of low frequencies is immediate from the definition of $E_{\lambda}^{s,l}(T)$.

Since we have to derive estimates uniform in time, we have to allow a slightly larger class of symbols following \cite{KochTataru2007}. This makes up for the difference between $E^{s,l}_{\lambda}(T)$ and $C([0,T],H^{s,l}_{\lambda})$. The proof of Proposition \ref{prop:shorttimeEnergyEstimateMBO} is concluded choosing symbols, which allow us to derive suitable estimates for frequency localized energies. For most of the time, we have $\xi_0=1$. To consider only high input frequencies of size $\gtrsim \xi_0$ will be important when constructing solutions (cf. Lemma \ref{lem:PrecompactnessApproximatingSequenceMBO}), as the estimates yield a smoothing factor.

To derive \eqref{eq:shorttimeEnergyEstimateMBO}, we analyze the following generalized energy $E_0^{a,\lambda}$ for a smooth, real-valued solution to \eqref{eq:mBO}:
\begin{equation*}
E^{a,\lambda}_0(u) = \int_{\xi_1 + \xi_2 = 0} a(\xi_1) \hat{u}(\xi_1) \hat{u}(\xi_2) d \Gamma_2^{\lambda} \; \left( := \frac{1}{\lambda} \sum_{\xi_1 \in \Z/\lambda} a(\xi_1) \hat{u}(\xi_1) \hat{u}(-\xi_1) \right).
\end{equation*}

The following symmetrization and integration by parts arguments can be found almost verbatim in \cite{Guo2011MBO} with the difference that the computations in \cite{Guo2011MBO} were carried out for a continuous frequency range.

We use the following notation for the $d-1$-dimensional grid in $d$-dimensional space:
\begin{equation*}
\Gamma^{\lambda}_d = \{ \xi_1 + \xi_2 + \ldots + \xi_d = 0 \; | \; \xi_i \in \Z/\lambda \},
\end{equation*}
and the measure is given as follows:
\begin{equation*}
\int_{\Gamma_d^\lambda} f(\xi_1) \ldots f(\xi_d) d\Gamma^\lambda_d(\xi_1,\ldots,\xi_d) = \frac{1}{\lambda^{d-1}} \sum_{\substack{\xi_1+\ldots+\xi_d = 0,\\ \xi_i \in \mathbb{Z}/\lambda}} f(\xi_1) \ldots f(\xi_d).
\end{equation*}

We find for the derivative of $E^{a,\lambda}_0(u(t))$ after symmetrization
\begin{equation*}
\begin{split}
\frac{d}{dt} E^{a,\lambda}_0(u) &= R^{a,\lambda}_4(u) \\
&=\frac{1}{2} \int_{\Gamma^\lambda_4} i[\xi_1 a(\xi_1) + \xi_2 a(\xi_2) + \xi_3 a(\xi_3) + \xi_4 a(\xi_4) ] \prod_{i=1}^4 \hat{u}(\xi_i) d\Gamma_4^{\lambda}.
\end{split}
\end{equation*}
Here, we have fixed the sign of the nonlinearity in \eqref{eq:mBO} for definiteness. From the below arguments follows that the sign is not relevant. Moreover, the symmetrization argument fails for differences of solutions. This leads to the well-known breakdown of uniform continuity of the data-to-solution mapping below $H^{1/2}$.

Next, we consider the correction term
\begin{equation*}
E^{a,\lambda}_1(u) = \int_{\Gamma^{\lambda}_4} b^{a}_4(\xi_1,\xi_2,\xi_3,\xi_4) \hat{u}(\xi_1) \hat{u}(\xi_2) \hat{u}(\xi_3) \hat{u}(\xi_4) d\Gamma_4^\lambda,
\end{equation*}
where we require the multiplier $b^a_4$ to satisfy the following identity on $\Gamma^\lambda_4$:
\begin{equation*}
 \left( \sum_{i=1}^4 \omega(\xi_i) \right) b^a_4(\xi_1,\xi_2,\xi_3,\xi_4) = \frac{-i}{2} \left( \sum_{i=1}^4 \xi_i a(\xi_i) \right).
\end{equation*}
Here, $\omega(\xi)=-\xi |\xi|$ denotes the dispersion relation of the Benjamin-Ono equation and the left-hand side of the above display corresponds to the contribution from the linear part of \eqref{eq:RescaledSolution} when replacing $\partial_t u$ in $\frac{d}{dt} E_1^{a,\lambda}$. This achieves a cancellation considering
\begin{equation*}
\begin{split}
R^{a,\lambda}_6(u) &= \frac{d}{dt} (E^{a,\lambda}_0(u) + E^{a,\lambda}_1(u)) \\
&= C \int_{\Gamma^\lambda_6} b^a_4(\xi_1,\xi_2,\xi_3,\xi_4+\xi_5+\xi_6)(\xi_4+\xi_5+\xi_6) \prod_{i=1}^6 \hat{u}(\xi_i).
\end{split}
\end{equation*}

We have the following proposition on choosing the multiplier $b^a_4$ smooth and extending it off diagonal, which allows us to separate variables easier later on. For the proof we follow ideas from  \cite{KochTataru2012} and \cite{ChristHolmerTataru2012}.
\begin{proposition}
\label{prop:energyMultiplierBounds}
Let $a \in S^s_{\varepsilon}$. Then, for each dyadic $\nu \leq \beta \leq \mu$, there is an extension $\tilde{b}^a_4$ of $b_4^a$ from the diagonal set
\begin{equation*}
\{ (\xi_1,\xi_2,\xi_3,\xi_4) \in \Gamma^{\lambda}_4: |\xi_1| \sim \nu, |\xi_2| \sim \beta, |\xi_3|,|\xi_4| \sim \mu \}
\end{equation*}
to the full dyadic set
\begin{equation*}
\{ (\xi_1,\xi_2,\xi_3,\xi_4) \in \mathbb{R}^4: |\xi_1| \sim \nu, |\xi_2| \sim \beta, |\xi_3|,|\xi_4| \sim \mu \},
\end{equation*}
which satisfies
\begin{equation*}
|\tilde{b}^a_4(\xi_1,\xi_2,\xi_3,\xi_4)| \lesssim a(\mu) \mu^{-1}
\end{equation*}
and
\begin{equation*}
| \partial_1^{\alpha_1} \partial_2^{\alpha_2} \partial_3^{\alpha_3} \partial_4^{\alpha_4} \tilde{b}^a_4(\xi_1,\xi_2,\xi_3,\xi_4)| \lesssim_{\alpha} a(\mu) \mu^{-1} \nu^{-\alpha_1} \beta^{-\alpha_2} \mu^{-(\alpha_3+\alpha_4)}.
\end{equation*}
with implicit constant depending on $\alpha$, but not on $\nu,\beta,\mu$.
\end{proposition}
\begin{proof}
In the following we can assume that $\omega(\xi_1)+\omega(\xi_2)+\omega(\xi_3)+\omega(\xi_4) \neq 0$ as long as we show $b_4$ to be smooth because it is easy to see that $\xi_1 a(\xi_1)+\xi_2 a(\xi_2) + \xi_3 a(\xi_3) + \xi_4 a(\xi_4)=0$  whenever $\omega(\xi_1)+\omega(\xi_2)+\omega(\xi_3)+\omega(\xi_4)=0$.

Furthermore, due to symmetry, we can assume that $\xi_3 >0$, $\xi_4<0$ and $|\xi_1| \leq |\xi_2| \leq |\xi_3| \leq |\xi_4|$, and by the support condition $\mu \gtrsim 1$.\\
\textbf{Case A:} $|\xi_2| \ll |\xi_3|$.\\
\underline{Subcase AI:} $\xi_1 > 0$ and $\xi_2 > 0$. We have
\begin{equation*} 
\omega(\xi_1) + \omega(\xi_2) + \omega(\xi_3) + \omega(\xi_4) = -2(\xi_1 \xi_2 + (\xi_1 + \xi_2) \xi_3),
\end{equation*}
 and we consider
\begin{equation*}
C b^{a}_4(\xi_1,\xi_2,\xi_3,\xi_4) = \frac{\xi_1 a(\xi_1) + \xi_2 a(\xi_2)}{\xi_1 \xi_2 + (\xi_2+\xi_1)\xi_3} + \frac{\xi_3 a(\xi_3) + \xi_4 a(\xi_4)}{(\xi_1 \xi_2 + \xi_2 \xi_3 + \xi_1 \xi_3)}.
\end{equation*}
The size and regularity properties of the first term follow from the size and regularity properties of $a$. For the second term we multiply with $1=-(\xi_1+\xi_2)/(\xi_3+\xi_4)$. We set
\begin{equation*}
q(\xi,\eta) = \frac{\xi a(\xi) + \eta a(\eta)}{ \xi + \eta},
\end{equation*}
which is a smooth function. Since $q$ satisfies the bounds 
\begin{equation*}
|\partial_\xi^m \partial_\eta^n q(\xi,\eta)| \lesssim a(N) N^{-(m+n)}
\end{equation*}
for $|\xi| \sim |\eta| \sim N$ and $m,n \in \mathbb{N}_0$, the conclusion follows also for the second term
\begin{equation*}
\frac{(\xi_1+\xi_2)(\xi_3 a(\xi_3) + \xi_4 a(\xi_4))}{(\xi_1 \xi_2 + \xi_2 \xi_3 + \xi_1 \xi_3)(\xi_3 + \xi_4)} = \frac{\xi_1+\xi_2}{\xi_1 \xi_2 + \xi_2 \xi_3 + \xi_1 \xi_3} q(\xi_3,\xi_4).
\end{equation*}
\underline{Subcase AII:} $\xi_1<0, \xi_2>0$. We find 
\begin{equation*}
\omega(\xi_1)+\omega(\xi_2)+\omega(\xi_3)+\omega(\xi_4)=-2(\xi_1+\xi_2)(\xi_1+\xi_3).
\end{equation*}
Hence,
\begin{equation*}
\begin{split}
C b^a_4(\xi_1,\xi_2,\xi_3,\xi_4) &= \frac{\xi_1 a(\xi_1) + \xi_2 a(\xi_2)}{(\xi_1+\xi_3)(\xi_1+\xi_2)} + \frac{\xi_3 a(\xi_3)+\xi_4 a(\xi_4)}{(\xi_1+\xi_3)(\xi_3+\xi_4)} \\
 &= \frac{1}{\xi_1+\xi_3} q(\xi_1,\xi_2) - \frac{1}{\xi_1+\xi_3} q(\xi_3,\xi_4),
 \end{split}
\end{equation*}
which satisfies the required bounds because $|\xi_1| \ll |\xi_3|$.\\
\textbf{Case B:} $|\xi_1| \ll |\xi_2| \sim |\xi_2| \sim |\xi_3|$.\\
After observing that
\begin{equation*}
|\omega(\xi_1) + \omega(\xi_2) + \omega(\xi_3) + \omega(\xi_4)| \sim |\xi_4|^2,
\end{equation*}
it is straight-forward to check that the extension
\begin{equation*}
\frac{\xi_1 a(\xi_1) + \xi_2 a(\xi_2) + \xi_3 a(\xi_3) + \xi_4 a(\xi_4)}{\omega(\xi_1) + \omega(\xi_2) + \omega(\xi_3) + \omega(\xi_4)}
\end{equation*}
suffices.\\
\textbf{Case C:} $|\xi_1| \sim |\xi_2| \sim |\xi_3| \sim |\xi_4|$.\\
We can assume $\xi_4 < 0, \xi_2 < 0$ and $\xi_1, \xi_3  >0$ and write
\begin{equation*}
\begin{split}
C b^a_4(\xi_1,\xi_2,\xi_3,\xi_4) &= \frac{a(\xi_1)\xi_1 + a(\xi_2) \xi_2}{(\xi_1+\xi_2)(\xi_2+\xi_3)} + \frac{a(\xi_3) \xi_3 + a(\xi_4) \xi_4}{(\xi_1+\xi_2)(\xi_2+\xi_3)} \\
&= \frac{q(\xi_1,\xi_2) - q(\xi_3,-\xi_1-\xi_2-\xi_3)}{\xi_2+\xi_3} \\
&= \frac{q(\xi_1,\xi_2) - q(\xi_1+(\xi_2+\xi_3),\xi_2-(\xi_2+\xi_3))}{\xi_2+\xi_3}.
\end{split}
\end{equation*}
Now the bounds follow from the size and regularity of $q$.
\end{proof}
 After smoothly extending the symbol on a dyadic scale $ \{ (\xi_1,\xi_2,\xi_3,\xi_4) \in \Gamma^\lambda_4: |\xi_1| \sim \nu, |\xi_2| \sim \beta, |\xi_3|,|\xi_4| \sim \mu \}$ off diagonal, we can separate variables without restriction (possibly after an additional partition of unity):
\begin{equation}
\label{eq:separatedVariablesMultiplierMBO}
\tilde{b}^a_4(\xi_1,\xi_2,\xi_3,\xi_4) \sim \tilde{b}^a_4(N_1,N_2,N_3,N_4) \chi_1(\xi_1) \chi_2(\xi_2) \chi_3(\xi_3) \chi_4(\xi_4)
\end{equation}
with regular bump functions $\chi$ of size $\lesssim 1$ localized at $|\xi_i| \lesssim N_i$ so that we can absorb the bump functions into the frequency projectors and invert the Fourier transform in space. The remaining expression in space and time will be estimated by short-time Strichartz estimates.

For details on the separation of variables by expanding $b_4^a$ into a rapidly converging Fourier series, see \cite[Section~5]{Hani2012} and in the short-time context \cite[Section~5]{ChristHolmerTataru2012}.

The boundary term $E^{a,\lambda}_1(u)$ can be estimated in a favorable way in terms of regularity. 
\begin{proposition}
\label{prop:energyBoundaryTermBoundMBO}
Let $0<l<1/4<s<1/2$ and $a \in S^s_{\varepsilon}$. Then, we find the estimate
\begin{equation*}
|E^{a,\lambda}_1(u(t))| \lesssim_{s,l} \Vert u(t) \Vert_{H^{s,l}_{\lambda}}^2 E^{a,\lambda}_0(u(t)).
\end{equation*}
the estimate to hold with implicit constant independent of $\lambda$.
\end{proposition}
\begin{proof}
We use a dyadic decomposition of $\Gamma^{\lambda}_4$ and the expansion \eqref{eq:separatedVariablesMultiplierMBO} to write
\begin{equation}
\label{eq:estimateE1}
\begin{split}
|E^{a,\lambda}_1(u)| &= \left| \int_{\Gamma^\lambda_4} \tilde{b}^a_4(\xi_1,\xi_2,\xi_3,\xi_4) \hat{u}(\xi_1) \hat{u}(\xi_2) \hat{u}(\xi_3) \hat{u}(\xi_4) d\Gamma^\lambda_4 \right| \\
&\leq \sum_{n_1 \leq n_2 \leq n_3 \leq n_4} \left| \int_{\Gamma^{\lambda}_4: |\xi_i| \sim 2^{n_i}} \tilde{b}^{a}_4(\xi_1,\xi_2,\xi_3,\xi_4) \hat{u}(\xi_1) \hat{u}(\xi_2) \hat{u}(\xi_3) \hat{u}(\xi_4) d\Gamma_4^\lambda \right| \\
&\lesssim \sum_{n_1 \leq n_2 \leq n_3 \leq n_4} |\tilde{b}^a_4(2^{n_1},2^{n_2},2^{n_3},2^{n_4})| \left|  \int_{\mathbb{\lambda T}} P_{n_1} u P_{n_2} u P_{n_3} u P_{n_4} u dx \right|.
\end{split}
\end{equation}
The normalization of $d \Gamma_4^\lambda$ allows us to return to position space with an estimate independent of $\lambda$. The size estimate of $b^a_4$ and applications of H\"older's and scale-invariant Bernstein's inequality allow for the continuation of \eqref{eq:estimateE1}
\begin{equation*}
\begin{split}
 &\lesssim \sum_{\substack{ n_1 \leq n_2 \leq n_3 \leq n_4,\\ n_4 \geq -2}} a(2^{n_4}) 2^{-n_4} \Vert P_{n_1} u \Vert_{L_x^{\infty}(\lambda \T)} \Vert P_{n_2} u \Vert_{L_x^{\infty}(\lambda \T)} \\
 &\quad \quad \times \Vert P_{n_3} u \Vert_{L_x^2(\lambda \T)} \Vert P_{n_4} u \Vert_{L_x^2(\lambda \T)} \\
&\lesssim \sum_{\substack{n_1 \leq n_2 \leq n_3 \leq n_4,\\ n_4 \geq -2}} a(2^{n_4}) \frac{2^{(n_1+n_2)/2}}{2^{n_4}} \prod_{i=1}^4 \Vert P_{n_i} u \Vert_{L_x^{2}(\lambda \T)} \\
&\lesssim \Vert u(t) \Vert_{H^{s,l}_{\lambda}}^2 E^{a,\lambda}_0(u),
\end{split}
\end{equation*}
which yields the claim.
\end{proof}
Now we estimate the remainder. Since the localization in time yields a behavior of solutions very similar to the real line case, some of the arguments from the proof below can be found in the corresponding proof on the real line \cite[Proposition~8.5.,~p.~1127]{Guo2011MBO}.
\begin{proposition}
\label{prop:remainderEnergyEstimateMBO}
Let $0<l<1/4<s $ and $T \in (0,1]$. There exists $\varepsilon = \varepsilon(s) >0$ and $\tilde{\varepsilon} (s) >0$ so that
\begin{equation}
\label{eq:remainderEnergyEstimateMBO}
\left| \int_{0}^T  R^{a,\lambda}_6(u) \right| \lesssim T \Vert u \Vert_{F^{s-\tilde{\varepsilon},l}_\lambda(T)}^6
\end{equation}
holds true for any $u \in C([-T,T],H^{\infty}_{\lambda})$ and $a \in S^{s}_{\varepsilon}$.
\end{proposition}
\begin{proof}
We have to estimate
\begin{equation}
\label{eq:reducedExpressionRemainderEstimateMBO}
\begin{split}
\int_0^T \int_{\Gamma^\lambda_6} [\tilde{b}^a_4(\xi_1,\xi_2,\xi_3,\xi_4+\xi_5+\xi_6)(\xi_4+\xi_5+\xi_6)] \prod_{j=1}^6 \hat{u}(t,\xi_j) d\Gamma_6^\lambda dt.
\end{split}
\end{equation}

Smoothly divide the frequencies into dyadic blocks and use the notation $|\xi_j| \sim 2^{k_j} = K_j$. Due to symmetry, we can assume that $K_1 \leq K_2 \leq K_3, \; \; K_4 \leq K_5 \leq K_6$. We write $\xi_{456} = \xi_4 + \xi_5 + \xi_6$. Temporarily, we introduce an additional frequency projector $\tilde{P}_K$ for $\xi_{456}$, which is also required to be smooth, and write $u_{k_i}$ for $P_{k_i} u$ in the following.

To reduce the Case-by-Case analysis, we suppose for the remainder of the argument that all frequencies under consideration are high, that means at least of size $1$. Due to favorable pointwise bounds and improved bilinear estimates (Remark \ref{rem:BilinearLowFrequencyRefinement}), the corresponding bounds for low frequencies are better, and thus omitted.

 We find
\begin{equation}
\label{eq:dyadicallyLocalizedEnergyEstimateMBO}
 \eqref{eq:reducedExpressionRemainderEstimateMBO} \lesssim \sum_{K_j,K} \left| \int_0^T \int_{\Gamma_6: |\xi_j| \sim K_j} \tilde{b}^a_4(\xi_1,\xi_2,\xi_3,\xi_{456}) \xi_{456} \chi_K(\xi_{456}) \prod_{j=1}^6 \hat{u}_j(t,\xi_j) dt \right|.
\end{equation}
We denote $\chi_K(\xi)=K^{-1} \chi(K^{-1} \xi)$, where $\chi$ denotes a suitable bump function. To derive estimates in terms of the short-time norms, we have to localize time reciprocally to the highest occuring frequency.
We bound the dyadically localized expression \eqref{eq:dyadicallyLocalizedEnergyEstimateMBO} in several cases:\\
\textbf{Case 1:} $K_5 \sim K_6 \sim K_{\max}, K_3 \lesssim K_5$: Write $C_1 = \{ (K_1,\ldots,K_6): \; K_5 \sim K_6 \sim K_{\max}, K_3 \lesssim K_5 \}$ and estimate this part of \eqref{eq:dyadicallyLocalizedEnergyEstimateMBO} by
 \begin{equation*}
 \begin{split}
 \sum_{K_j \in C_1, K \lesssim K_3} T K_6 \sup_{|I| = K_6^{-1}} | \int_{I} \int_{\Gamma^{\lambda}_6} &\tilde{b}^a_4(\xi_1,\xi_2,\xi_3,\xi_{456}) \xi_{456} \\
 &\quad \times \chi_K(\xi_{456}) \prod_{j=1}^6 \hat{u}(t,\xi_j) d\Gamma_6^{\lambda} dt |,
 \end{split}
 \end{equation*}
 where $I \subseteq [0,T]$ denotes an interval of length $K_6^{-1}$.\\
 We expand
\begin{equation}
\label{eq:FourierIntegralMBO}
\chi_K(\xi_{456}) = \int_{\R} e^{-ix \xi_{456}} f_K(x) dx = \int_{\R} e^{-ix \xi_4} e^{-ix \xi_5} e^{-ix \xi_6} f_K(x) dx ,
\end{equation}
and from the definition of $\chi_K$ we find for $f_K$ a uniformly in $K$ bounded $L_x^1$-norm.

Plugging in the expression \eqref{eq:separatedVariablesMultiplierMBO} and absorbing the factors coming from \eqref{eq:FourierIntegralMBO} into the $\hat{u}_i$\footnote{Note that the $U^2_{BO} L^2_\lambda$-norm is not effected by factors with bounded modulus (cf. \cite[Section~5]{ChristHolmerTataru2012}).}, we are left with estimating
\begin{equation*}
\begin{split}
&\quad \sum_{\underline{K} \in C_1, K \lesssim K_3} T K_6 |\tilde{b}_4^a(K_1,K_2,K_3,K) K | \sup_{|I| = K_6^{-1}} | \int_I \int_{\Gamma_6^\lambda} \prod_{j=1}^6 \hat{u}_{k_j}(t,\xi_j) d\Gamma_6^\lambda dt | \\
&\lesssim T \sum_{ \underline{K} \in C_1, K \lesssim K_3} |\tilde{b}^a_4(K_1,K_2,K_3,K) K| K_6 \\
&\quad \quad \times \sup_{|I| = K_6^{-1}}  | \int_{I} \int_{\lambda \T} u_{k_1} \ldots u_{k_6} dx dt |,
\end{split}
\end{equation*}
where we have changed back to position space at last. Using the pointwise estimate for $\tilde{b}^a_4$, we find
\begin{equation*}
\sum_{K \lesssim K_3} |b^a_4(K_1,K_2,K_3,K) K| \lesssim a(K_3).
\end{equation*}
Next, we use the short-time estimates from Section \ref{section:shorttimeLinearEstimates} to derive suitable estimates for the expression
\begin{equation}
\label{eq:ReducedRemainderPositionSpaceMBO}
\left| \int_{I} dt \int_{\lambda \T} dx P_{k_1} u_1 \ldots P_{k_6} u_6 \right|.
\end{equation}
The bounds for \eqref{eq:ReducedRemainderPositionSpaceMBO} are derived according to the separation of the involved frequencies. Let $K_1^* \leq \ldots \leq K_6^*$ denote the increasing rearrangement of $\{K_1,\ldots,K_6\}$, and similarly for $\{k_1,\ldots,k_6 \}$.\\
\underline{Subcase 1a:} $K_4^* \ll K_6^*$, $K_3^* \ll K_4^* \sim K_6^*$: In this case we can use two bilinear Strichartz estimates. Say $K_1$ and $K_2$ are the lowest and second-to-lowest frequencies and $K_4,K_5,K_6$ correspond to $K_4^*,K_5^*,K_6^*$. Following Remark \ref{rem:IntervalSeparation}, we arrange $u_{k_4} u_{k_5}$ and $u_{k_3} u_{k_6}$ in pairs for two bilinear Strichartz estimates and use Bernstein's inequality on $u_{k_1}$ and $u_{k_2}$. We find by Proposition \ref{prop:StrichartzEstimatesMBO}
\begin{equation*}
\begin{split}
\eqref{eq:ReducedRemainderPositionSpaceMBO} &\lesssim \Vert u_{k_1} \Vert_{L_t^\infty L_x^\infty(\lambda \T)} \Vert u_{k_2} \Vert_{L_t^\infty L_x^\infty(\lambda \T)} \Vert u_{k_4} u_{k_5} \Vert_{L_t^2 L_x^2(\lambda \T)} \Vert u_{k_3} u_{k_6} \Vert_{L_t^2 L_x^2(\lambda \T)} \\
&\lesssim \frac{(K_1^* K_2^*)^{1/2}}{K_6^*} \Vert u_{k_1} \Vert_{F_{k_1,\lambda}} \Vert u_{k_2} \Vert_{F_{k_2,\lambda}} \ldots \Vert u_{k_6} \Vert_{F_{k_6,\lambda}}. 
\end{split}
\end{equation*}
Gathering the estimates, we have proved
\begin{equation*}
\begin{split}
 \left| \int_0^T R^{a,\lambda}_{6,C_{1a}}(u) \right| &\lesssim T \sum_{K_i} (K_4^*)^{2(s+\varepsilon)} (K_1^* K_2^*)^{1/2} \prod_{i=1}^6 \Vert u_{k_i} \Vert_{F_{k_i,\lambda}} \\
&\lesssim T \prod_{i=1}^6 \Vert u \Vert_{F^{s-\tilde{\varepsilon},l}_{\lambda}(T)},
\end{split}
\end{equation*}
where the last step follows from carrying out the summations and choosing $\varepsilon$ and $\tilde{\varepsilon}$ sufficiently small.\\
\underline{Subcase 1b}: $K_2^* \ll K_3^* \sim K_6^*$, $K_1^* \ll K_2^* \sim K_6^*$: In this case we use a bilinear estimate on $u_{k_2^*} u_{k_6^*}$, three linear $L^6_{t,x}$-Strichartz estimates on $u_{k_3^*}$, $u_{k_4^*}$, $u_{k_5^*}$ and one pointwise bound $u_{k_1^*}$. This is clearly possible if $K_2^* \ll K_3^*$. If only $K_1^* \ll K_2^*$, this follows from a similar argument as in Remark \ref{rem:IntervalSeparation} because the number of functions with comparable frequency sizes is odd.\\
This gives
\begin{equation*}
\begin{split}
|R^{a,\lambda}_{6,C_{1b}}(u)| &\lesssim T \sum_{K_1^* \ll K_2^* \sim K_6^*} (K_1^* K_6^*)^{1/2} (K_6^*)^{2(s+\varepsilon)} \prod_{i=1}^6 \Vert u_{k_i} \Vert_{F_{k_i,\lambda}} \\
&\lesssim T \prod_{i=1}^6 \Vert u \Vert_{F^{s-\tilde{\varepsilon},l}_{\lambda}(T)}.
\end{split}
\end{equation*}
\underline{Subcase 1c:} $K_1^* \sim K_6^*$. Here, no multilinear estimates are used, but six $L^6_{t,x}$-Strichartz estimates to find
\begin{equation*}
| \int_0^T R_{6,C_{1c}}^{a,\lambda}(u)| \lesssim T \sum_{K_1^* \sim K_6^*} (K_6^*)^{2(s+\varepsilon)} K_6^* \prod_{i=1}^6 \Vert u_{k_i} \Vert_{F_{k_i,\lambda}} \lesssim T \prod_{i=1}^6 \Vert u \Vert_{F_\lambda^{s-\tilde{\varepsilon},l}(T)}.
\end{equation*}
\textbf{Case 2}: $K_2 \sim K_3 \sim K_{\max}, K_6 \lesssim K_2$: Introduce the notation
\begin{equation*}
C_2 = \{(K_1,\ldots,K_6) \; | \; K_2 \sim K_3 \sim K_{\max}, K_6 \lesssim K_2 \}
\end{equation*}
 and suppose in the following $K \lesssim K_6$. We have to bound
\begin{equation}
\label{eq:reducedRemainderEstimateIIMBO}
\begin{split}
T \sum_{\underline{K} \in C_2, K \lesssim K_6} K_3 \sup_{|I| = K_3^{-1}} | \int_I \int_{\Gamma_6^\lambda} &\tilde{b}_4^a(\xi_1,\xi_2,\xi_3,\xi_{456}) \xi_{456} \chi_K(\xi_{456})  \\
&\quad \times \prod_{j=1}^6 \hat{u}_{k_j}(t,\xi_j) d\Gamma_6^\lambda dt |.
\end{split}
\end{equation}
Following along the above lines, we are led to the estimate
\begin{equation*}
\begin{split}
\eqref{eq:reducedRemainderEstimateIIMBO} &\lesssim T \sum_{\underline{K} \in C_2, K \lesssim K_6} |b_4(K_1,K_2,K_3,K) K| K_3 \sup_{|I|=K_3^{-1}} \left| \int_I \int_{\lambda \T} u_{k_1} \ldots u_{k_6} dx dt \right| \\
&\lesssim T \sum_{\underline{K} \in C_2} K_6 K_3^{2(s+\varepsilon)} \sup_{|I|=K_3^{-1}} \left| \int_{I} \int_{\lambda \T} u_{k_1} \ldots u_{k_6} dx dt \right|.
\end{split}
\end{equation*}
The product is estimated according to the separation of the frequencies like in Case 1.\\
\underline{Subcase 2a:} $K_4^* \ll K_5^* \sim K_6^*, \quad K_3^* \ll K_4^* \sim K_6^*$: Here, we can use two bilinear Strichartz estimates on the highest frequencies leading to a gain of $(K_6^*)^{-1}$ and two pointwise bounds on the lowest frequencies, which gives a factor $(K_1^* K_2^*)^{1/2}$. Summation yields
\begin{equation*}
T \sum_{K_i^*} K_4^* (K_6^*)^{2(s+\varepsilon)} (K_1^* K_2^*)^{1/2} (K_6^*)^{-1} \prod_{i=1}^6 \Vert u_{k_i} \Vert_{F_{k_i,\lambda}} \lesssim T \prod_{i=1}^6 \Vert u \Vert_{F_\lambda^{s-\tilde{\varepsilon},l}(T)}.
\end{equation*}
\underline{Subcase 2b:} $K_2^* \ll K_3^* \sim K_6^*$, $K_1^* \ll K_2^* \sim K_6^*$: In this case one uses again one bilinear estimate, three $L^6_{t,x}$-Strichartz estimates and one pointwise bound to find
\begin{equation*}
\begin{split}
T \sum_{K_i^*} (K_1^*)^{1/2} K_6^* (K_6^*)^{-1/2} (K_6^*)^{2(s+\varepsilon)} \prod_{i=1}^6 \Vert u_{k_i} \Vert_{F_{k_i,\lambda}} \lesssim T \prod_{i=1}^6 \Vert u \Vert_{F_\lambda^{s-\tilde{\varepsilon},l}(T)}.
\end{split}
\end{equation*}
\underline{Subcase 2c}: $K_1^* \sim K_6^*$: After using six $L^6_{t,x}$-Strichartz estimates, the estimate is concluded like in Subcase 1c.\\
\textbf{Case 3}: $K_3 \sim K_6 \sim K_{max}; \; K_2, K_5 \ll K_3$: In this case the above argument is enhanced with an additional symmetrization, which corresponds to a further integration by parts. Note that
\begin{equation}
\begin{split}
\label{eq:SecondSymmetrizationMBO}
&\quad \int_0^T \int_{\Gamma_6^\lambda} b_4^a(\xi_1,\xi_2,\xi_3,\xi_{456}) \xi_{456} \prod_{i=1}^6 \hat{u}(t,\xi_i) d\Gamma_6^\lambda dt \\
&= \int_0^T \int_{\Gamma_6^\lambda} b_4^a(\xi_4,\xi_5,\xi_6,\xi_{123}) \xi_{123} \prod_{i=1}^6 \hat{u}(t,\xi_i) d\Gamma_6^\lambda dt \\
&= \int_0^T \int_{\Gamma_6^\lambda} -b_4^a(\xi_4,\xi_5,\xi_6,-\xi_{456}) \xi_{456} \prod_{i=1}^6 \hat{u}(t,\xi_i) d\Gamma_6^\lambda dt \\
&= - \int_0^T \int_{\Gamma_6^\lambda} b_4^a(-\xi_4,-\xi_5,-\xi_6,\xi_{456}) \xi_{456} \prod_{i=1}^6 \hat{u}(t,\xi_i) d\Gamma_6^\lambda dt.
\end{split}
\end{equation}
Thus, it is enough to estimate
\begin{equation*}
\int_0^T \int_{\Gamma_6^\lambda} [b_4^a(\xi_1,\xi_2,\xi_3,\xi_{456}) - b_4^a(-\xi_4,-\xi_5,-\xi_6,\xi_{456})] \xi_{456} \prod_{i=1}^6 \hat{u}(t,\xi_i) d\Gamma_6^\lambda dt.
\end{equation*}

By the mean value theorem and regularity of $\tilde{b}^a_4$, we find the symmetrized expression to satisfy the same regularity assumptions like $\tilde{b}_4^a$. Further, note the size estimate
\begin{equation*}
|\tilde{b}_4^a(\xi_1,\xi_2,\xi_3,\xi_{456}) - \tilde{b}_4^a(-\xi_4,-\xi_5,-\xi_6,\xi_{456})| \lesssim \frac{(K_6^*)^{2(s+\varepsilon)}}{(K_6^*)^2} K_4^*.
\end{equation*}

As $K_4^* \ll K_5^* \sim K_6^*$, we can use two bilinear Strichartz estimates and two pointwise bounds as in the above Subcases 1a, 2a to finish the proof.
\end{proof}
To conclude the proof of the energy estimate, we derive bounds for the frequency localized energy. Recall the following lemma from \cite{KochTataru2007}, which was only proved on the real line; however, the proof carries over to $\lambda \T$.
\begin{lemma}{\cite[Lemma~5.5.,~p.~26]{KochTataru2007}}
\label{lem:EnergyThresholdSymbols}
For any $u_0 \in H_{\lambda}^{s,l}$ and $\varepsilon >0$, there is a sequence $(\beta_n)_{n \in \mathbb{N}}$ satisfying the following conditions:
\begin{enumerate}
\item[(a)] For any $n \in \mathbb{N}$, we have $2^{2ns} \Vert P_n u_0 \Vert_{L^2(\lambda \T)}^2 \leq \beta_n \Vert u_0 \Vert_{H_{\lambda}^s}^2 $,
\item[(b)] $\sum_n \beta_n \lesssim 1$,
\item[(c)] $(\beta_n)$ satisfies a log-Lipschitz condition, which is given by
\begin{equation*}
| \log_2 \beta_{n} - \log_2 \beta_{m} | \leq \frac{\varepsilon}{2} |n - m|.
\end{equation*}
\end{enumerate}
\end{lemma}
By this, we finish the proof of Proposition \ref{prop:shorttimeEnergyEstimateMBO} in a similar spirit to \cite[Section~5]{KochTataru2007}.
\begin{proof}[Proof of Proposition \ref{prop:shorttimeEnergyEstimateMBO}]
We choose $\varepsilon>0$ and $\tilde{\varepsilon}>0$ in dependence of $s>1/4$ so that \eqref{eq:remainderEnergyEstimateMBO} becomes true for any $a \in S^s_{\varepsilon}$ by virtue of Proposition \ref{prop:remainderEnergyEstimateMBO}.
 
Let $k_0 \in  \mathbb{N}$ and $(\beta_n)$ be an envelope sequence from Lemma \ref{lem:EnergyThresholdSymbols} for the initial data $u_0$. We prove
\begin{equation}
\label{eq:energyThresholdBoundMBO}
\sup_{t \in [-T,T]} 2^{2ks} \Vert P_k u(t) \Vert_{L^2(\lambda \T)}^2 \lesssim \beta_k(\Vert u_0 \Vert_{H_{\lambda}^{s,l}}^2 + T \Vert u \Vert_{F_{\lambda}^{s-\tilde{\varepsilon},l}(T)}^6 ),
\end{equation}
from which follows \eqref{eq:shorttimeEnergyEstimateMBO} after carrying out the summation over $k$, due to property (b) from Lemma \ref{lem:EnergyThresholdSymbols}.

 We consider $\tilde{a}^{k_0}_k= 2^{2ks} \max(1,\beta_{k_0}^{-1} 2^{-\varepsilon |k-k_0|})$, and we find
\begin{equation*}
\begin{split}
\sum_{k \geq 1} \tilde{a}^{k_0}_k \Vert P_k u_0 \Vert_{L^2(\lambda \T)}^2 &\leq \sum_{k} 2^{2ks} \Vert P_k u_0 \Vert_{L^2(\lambda \T)}^2 + 2^{2ks} 2^{-\frac{\varepsilon}{2} |k-k_0|} \beta_{k}^{-1} \Vert P_k u_0 \Vert_{L^2(\lambda \T)}^2 \\
&\lesssim_{\varepsilon} \Vert u_0 \Vert_{H_{\lambda}^{s,l}}^2,
\end{split}
\end{equation*} 
due to the slowly varying condition and property (a) from Lemma \ref{lem:EnergyThresholdSymbols}.

The implicit constant in the estimate above does not depend on $k_0$, but only on $\varepsilon$. Smoothing out a linearly interpolated version, we can find a symbol $a^{k_0}(\xi) \in S^s_{\varepsilon}$ so that
\begin{equation*}
a^{k_0}(\xi) \sim \tilde{a}^{k_0}_k, \; \; |\xi| \sim 2^{k}.
\end{equation*}
For details on this procedure, see \cite[Subsection~2.3]{OhWang2018}.

 Next, following the computations from the beginning of this subsection, we find by means of Proposition \ref{prop:energyBoundaryTermBoundMBO} and \ref{prop:remainderEnergyEstimateMBO}
\begin{equation*}
\Vert u(t) \Vert_{H_\lambda^a}^2 \lesssim_{s} \Vert u_0 \Vert_{H_\lambda^a}^2 + \Vert u_0 \Vert_{H^{s,l}_\lambda}^2 \Vert u_0 \Vert_{H_\lambda^a}^2 + \Vert u(t) \Vert_{H^{s,l}_\lambda}^2 \Vert u(t) \Vert_{H_\lambda^a}^2 + T \Vert u \Vert_{F^{s-\tilde{\varepsilon},l}_\lambda(T)}^6.
\end{equation*}
Requiring $\Vert u(t) \Vert_{H^{s,l}_{\lambda}}$ to be small, this implies
\begin{equation*}
\Vert u(t) \Vert_{H_\lambda^a}^2 \lesssim_{s} \Vert u_0 \Vert_{H_\lambda^a}^2 + T \Vert u \Vert^6_{F^{s-\tilde{\varepsilon},l}_\lambda(T)} \lesssim_\varepsilon \Vert u_0 \Vert_{H_\lambda^{s,l}}^2 + T \Vert u \Vert^6_{F_{\lambda}^{s-\tilde{\varepsilon},l}(T)}
\end{equation*}
with the second estimate following from $\Vert u_0 \Vert_{H_\lambda^a}^2 \lesssim_{\varepsilon} \Vert u_0 \Vert_{H_\lambda^{s,l}}^2$. At last, since $\Vert u \Vert_{H_\lambda^a}^2 \sim \sum_{k \geq 1} \tilde{a}^{k_0}_k \Vert P_k u(t) \Vert_{L^2(\lambda \T)}^2$, we arrive at
\begin{equation*}
\sup_{t \in [0,T]} \left( \sum_{k \geq 0} \tilde{a}^{k_0}_k \Vert P_k u(t) \Vert_{L^2(\lambda \T)}^2 \right) \lesssim_s \Vert u_0 \Vert_{H_\lambda^{s,l}}^2 + T \Vert u \Vert^6_{F^{s-\tilde{\varepsilon},l}_{\lambda}(T)}.
\end{equation*}
Restricting the sum to $k_0$ implies \eqref{eq:energyThresholdBoundMBO}. The proof is complete.
\end{proof}

\section{Modifications for the derivative nonlinear Schr\"odinger equation}
\label{section:ModificationsDNLS}
In this paragraph we sketch the necessary modifications to show that the assertions from Theorem \ref{thm:mainResultMBO} on periodic solutions to \eqref{eq:mBO} extend to periodic solutions to \eqref{eq:dNLS}.

We show a corresponding short-time trilinear estimate and an energy estimate with smoothing effect after adapting the short-time function spaces to the Schr\"odinger flow. This is suppressed in the notation in the following. We show in addition to the linear estimate from Lemma \ref{lem:ShorttimeEnergyEstimate}
\begin{equation*}
\left\{ \begin{array}{cl}
\Vert \partial_x (|u|^2 u) \Vert_{N^{s,l}_{\lambda}(T)} &\lesssim \Vert u \Vert^3_{F^{s,l}_{\lambda}(T)} \\
\Vert u \Vert^2_{E^{s,l}_{\lambda}(T)} &\lesssim \Vert u_0 \Vert^2_{H^{s,l}_{\lambda}} + T ( \Vert u \Vert^6_{F^{s-\varepsilon,l}_\lambda(T)} + \Vert u \Vert^8_{F_{\lambda}^{s-\varepsilon,l}(T)} )
\end{array} \right.
\end{equation*}
provided that $\lambda \geq 1$, $0<l<1/4<s$ and $\varepsilon = \varepsilon(s) > 0$, $\sup_{t \in [-T,T]} \Vert u(t) \Vert_{H^{s,l}_\lambda} \leq \delta(s)$. Like above, we suppose without loss of generality that $\int_{\mathbb{T}} u_0 dx = 0$. We start with the short-time trilinear estimate:
\begin{proposition}
\label{prop:shorttimeTrilinearEstimateDNLS}
Let $\lambda \geq 1$, $T \in (0,1]$, $0<l<1/4<s<1/2$ and suppose that $u,v,w \in F^{s,l}_{\lambda}(T)$. Then, we find the following estimate to hold:
\begin{equation}
\label{eq:ShorttimeTrilinearEstimateDNLS}
\Vert \partial_x (u \overline{v} w) \Vert_{N^{s,l}_{\lambda}(T)} \lesssim \Vert u \Vert_{F^{s,l}_{\lambda}(T)} \Vert v \Vert_{F^{s,l}_{\lambda}(T)} \Vert w \Vert_{F^{s,l}_{\lambda}(T)}.
\end{equation}
\end{proposition}
\begin{proof}
The strategy is the same as in the proof of Proposition \ref{prop:shorttimeTrilinearEstimateMBO}. The claim follows from revisiting the proof of Proposition \ref{prop:shorttimeTrilinearEstimateMBO}, and whenever one applies an estimate from Proposition \ref{prop:StrichartzEstimatesMBO}, the corresponding estimate from Proposition \ref{prop:StrichartzEstimatesSEQ} is applied.

Recall the possible frequency interactions, which were enumerated for the proof of Proposition \ref{prop:shorttimeTrilinearEstimateMBO} and remain the same. We give the details in case of $High \times Low \times Low \rightarrow High$-interaction and $High \times High \times High \rightarrow High$-interaction. In the first case, under the same assumptions as in Lemma \ref{lem:HighLowLowHighInteractionMBO}, let $I$ be an interval of length $2^{-k_4}$ and we compute by H\"older in time, a short-time bilinear Strichartz estimate and Bernstein's inequality
\begin{equation*}
\begin{split}
\Vert P_{k_4} \partial_x (u_{k_1} \overline{v}_{k_2} w_{k_3}) \Vert_{DU^2_{\Delta}(I;L^2_x(\lambda \T))} &\lesssim 2^{k_4} \Vert u_{k_1} \overline{v}_{k_2} w_{k_3} \Vert_{L_t^1(I;L^2_x(\lambda \T))} \\
&\lesssim 2^{k_4/2} \Vert u_{k_1} \Vert_{L_t^\infty(I;L^\infty(\lambda \T))} \Vert \overline{v}_{k_2} w_{k_3} \Vert_{L_t^2(I;L^2_x(\lambda \T))} \\
&\lesssim 2^{k_1/2} \min(1,2^{k_2/2}) \Vert u_{k_1} \Vert_{U^2_{\Delta}(I;L^2_x(\lambda \T))} \\
&\quad \quad \times  \Vert v_{k_2} \Vert_{U^2_{\Delta}(I;L^2_x(\lambda \T))} \Vert w_{k_3} \Vert_{U^2_{\Delta}(I;L^2_x(\lambda \T))}.
\end{split}
\end{equation*}
By the above means, the corresponding estimate to Lemma \ref{lem:HighLowLowHighInteractionMBO} follows from the definition of the function spaces.

In the second case we use H\"older in time and three $L^6_{t,x}$-Strichartz estimates to find
\begin{equation*}
\begin{split}
&\quad \Vert P_{k_4} \partial_x (u_{k_1} \overline{v}_{k_2} w_{k_3}) \Vert_{DU^2_{\Delta}(I;L^2_x(\lambda \T))} \\
&\lesssim 2^{k_4} \Vert u_{k_1} \overline{v}_{k_2} w_{k_3} \Vert_{L_t^1(I;L^2_x(\lambda \T))} \\
&\lesssim 2^{k_4/2} \Vert u_{k_1} \overline{v}_{k_2} w_{k_3} \Vert_{L_t^2(I;L^2_x(\lambda \T))} \\
&\lesssim 2^{k_4/2} \Vert u_{k_1} \Vert_{L_t^6(I;L_x^6(\lambda \T))} \Vert v_{k_2} \Vert_{L_t^6(I;L_x^6(\lambda \T))} \Vert w_{k_3} \Vert_{L_t^6(I;L_x^6(\lambda \T))} \\
&\lesssim 2^{k_4/2} \Vert u_{k_1} \Vert_{U^2_{\Delta}(I;L^2_x(\lambda \T))} \Vert v_{k_2} \Vert_{U^2_{\Delta}(I;L^2_x(\lambda \T))} \Vert w_{k_3} \Vert_{U^2_{\Delta}(I;L^2_x(\lambda \T))}.
\end{split}
\end{equation*}
Also the other cases follow as in Section \ref{section:ShorttimeTrilinearEstimateMBO}.
\end{proof}
The energy estimate is more involved due to the reduced symmetry. If one wants to stick to the use of linear and bilinear short-time Strichartz estimates, one has to integrate by parts a second time in one case of the remainder estimate. Alternatively, the claim follows from a refined trilinear estimate. Below we do the extra work of a second integration by parts to point out that the second correction satisfies better bounds, at least in this specific case.
\begin{proposition}
\label{prop:EnergyEstimateDNLS}
Let $T \in (0,1]$, $0<l<1/4<s$ and suppose that \\
$u \in C([-T,T],H^\infty_{\lambda})$ is a smooth solution to \eqref{eq:dNLS}. Then, there exists $\tilde{\varepsilon}(s)$ and $\delta(s)>0$ such that we find the estimate
\begin{equation*}
\Vert u \Vert^2_{E^{s,l}_{\lambda}(T)} \lesssim_s \Vert u_0 \Vert_{H_\lambda^{s,l}}^2 + T(\Vert u \Vert^6_{F^{s-\tilde{\varepsilon},l}_\lambda(T)} + \Vert u \Vert^8_{F^{s-\tilde{\varepsilon},l}_\lambda(T)})
\end{equation*}
to hold provided that
\begin{equation*}
\sup_{t \in [-T,T]} \Vert u(t) \Vert_{H^{s,l}_{\lambda}} \leq \delta(s).
\end{equation*}
\end{proposition}
We analyze the following generalized energy $E_0^{a,\lambda}$ for a smooth solution to \eqref{eq:dNLS}:
\begin{equation*}
E_0^{a,\lambda}(u) = \int_{\Gamma_2^\lambda} a(\xi_1) \hat{u}(\xi_1) \hat{\overline{u}}(\xi_2) d\Gamma_2^\lambda.
\end{equation*}

In the following we carry out the program from Section \ref{section:EnergyEstimatesMBO}. We have to take into account the change of dispersion relation and that the solutions are no longer real-valued. It turns out that the symmetrized expression when computing $\frac{d}{dt} E_0^{a,\lambda}$ is still close to the corresponding expression from Section \ref{section:EnergyEstimatesMBO}:
\begin{equation*}
\begin{split}
\frac{d}{dt} E_0^{a,\lambda} &= \int_{\xi_1+\xi_2=0} a(\xi_1) (i \xi_1) \int_{\xi_1 = \xi_{11}+\xi_{12}+\xi_{13}} \hat{u}(\xi_{11}) \hat{\overline{u}}(\xi_{12}) \hat{u}(\xi_{13}) d\Gamma_3^\lambda \hat{\overline{u}}(\xi_2) d\Gamma_2^\lambda \\
&\quad + \int_{\xi_1+\xi_2=0} a(\xi_1) \hat{u}(\xi_1) (i \xi_2) \int_{\xi_2 = \xi_{21}+\xi_{22} + \xi_{23}} \hat{\overline{u}}(\xi_{21}) \hat{u}(\xi_{22}) \hat{\overline{u}}(\xi_{23}) d\Gamma_3^\lambda d\Gamma_2^\lambda \\
&= -\frac{i}{2} \int_{\Gamma_4^\lambda} \left( \sum_{i=1}^4 a(\xi_i) \xi_i \right) \hat{u}(\xi_1) \hat{\overline{u}}(\xi_2) \hat{u}(\xi_3) \hat{\overline{u}}(\xi_4) d\Gamma_4^\lambda.
\end{split}
\end{equation*}
As above, we consider the correction term
\begin{equation*}
E_1^{a,\lambda}(u) = \int_{\Gamma_4^\lambda} b_4^a(\xi_1,\xi_2,\xi_3,\xi_4) \hat{u}(\xi_1) \hat{\overline{u}}(\xi_2) \hat{u}(\xi_3) \hat{\overline{u}}(\xi_4) d\Gamma_4^\lambda,
\end{equation*}
and we require the multiplier $b_4^a$ to satisfy the following identity on $\Gamma_4^\lambda$:
\begin{equation}
\label{eq:correctionTermIdentity}
(-i)(\xi_1^2 - \xi_2^2 + \xi_3^2 - \xi_4^2) b_4^a(\xi_1,\xi_2,\xi_3,\xi_4) = \frac{i}{2} ( \xi_1 a(\xi_1) + \xi_2 a(\xi_2) + \xi_3 a(\xi_3) + \xi_4 a(\xi_4)).
\end{equation}

This yields
\begin{equation*}
\begin{split}
R_6^{a,\lambda} &= \frac{d}{dt}(E_0^{a,\lambda} + E_1^{a,\lambda})\\
&= 2 \int_{\Gamma_6^\lambda} b_4^a(\underbrace{\xi_{11}+\xi_{12}+\xi_{13}}_{\xi_1},\xi_2,\xi_3,\xi_4) (i\xi_1) \hat{u}(\xi_{11}) \hat{\overline{u}}(\xi_{12}) \hat{u}(\xi_{13}) \\
&\quad \quad \times \hat{\overline{u}}(\xi_2) \hat{u}(\xi_3) \hat{\overline{u}}(\xi_4) d\Gamma_6^\lambda \\
&\quad + 2 \int_{\Gamma_6^\lambda} b_4^a(\xi_1,\underbrace{\xi_{21}+\xi_{22}+\xi_{23}}_{\xi_2},\xi_3,\xi_4) \hat{u}(\xi_1) \\
&\quad \quad \times (i\xi_2) \hat{\overline{u}}(\xi_{21}) \hat{u}(\xi_{22}) \hat{\overline{u}}(\xi_{23}) \hat{u}(\xi_3) \hat{\overline{u}}(\xi_4) d\Gamma_6^\lambda \\
&= C \Im \int_{\Gamma_6^\lambda} b_4^a(\xi_{11}+\xi_{12}+\xi_{13}, \xi_2, \xi_3, \xi_4) \xi_1 \hat{u}(\xi_{11}) \hat{\overline{u}}(\xi_{12}) \hat{u}(\xi_{13}) \\
&\quad \quad \times \hat{\overline{u}}(\xi_2) \hat{u}(\xi_3) \hat{\overline{u}}(\xi_4) d\Gamma_6^\lambda. 
\end{split}
\end{equation*}

We show the same size and regularity estimates for the symbol $b_4^a$ from \eqref{eq:correctionTermIdentity} as in Section \ref{section:EnergyEstimatesMBO}.
\begin{proposition}
\label{prop:SymbolRegularityDNLS}
Let $a \in S^s_{\varepsilon}$. Then, for each dyadic $\lambda \leq \beta \leq \mu$, there is an extension $\tilde{b_4^a}$ of $b_4^a$ from the diagonal set
\begin{equation*}
\{(\xi_1,\xi_2,\xi_3,\xi_4) \in \Gamma_4^\lambda \; | \; |\xi_1^*| \sim \lambda, |\xi_2^*| \sim \beta, |\xi_3^*| \sim |\xi_4^*| \sim \mu \}
\end{equation*}
to the full dyadic set
\begin{equation*}
\{(\xi_1,\xi_2,\xi_3,\xi_4) \in \R^4 \; | \; |\xi_1^*| \sim \lambda, |\xi_2^*| \sim \beta, |\xi_3^*| \sim |\xi_4^*| \sim \mu \},
\end{equation*}
which satisfies
\begin{equation*}
|\tilde{b}_4^a| \lesssim a(\mu) \mu^{-1}
\end{equation*}
and
\begin{equation*}
|\partial_1^{\alpha_1} \partial_2^{\alpha_2} \partial_3^{\alpha_3} \partial_4^{\alpha_4} \tilde{b}_4^a| \lesssim_{\alpha} a(\mu) \mu^{-1} N_1^{-\alpha_1} N_2^{-\alpha_2} N_3^{-\alpha_3} N_4^{-\alpha_4},
\end{equation*}
where $|\xi_i| \sim N_i$ and $|\xi_1^*|\leq \ldots \leq |\xi_4^*|$ denotes an increasing rearrangement of $\xi_i$, $i=1,\ldots,4$.
\end{proposition}
\begin{proof}
We prove the proposition through Case-by-Case analysis. Note the symmetries between $\xi_1$ and $\xi_3$, $\xi_2$ and $\xi_4$ and the pairs $\{ \xi_1,\xi_3 \}$ and $\{ \xi_2,\xi_4 \}$. Moreover, we dispose of irrelevant factors below.\\
\textbf{Case 1} $|\xi_3^*| \ll |\xi_1^*|$:\\
\underline{Subcase 1a} $|\xi_1| \sim |\xi_2| \gg \max( |\xi_3|,|\xi_4| )$:\\
In this subcase we find $|\xi_2+\xi_3| \sim |\xi_1|$ and decompose
\begin{equation*}
b_4^a(\xi_1,\xi_2,\xi_3,\xi_4) = \frac{\xi_1 a(\xi_1) + \xi_2 a(\xi_2)}{(\xi_1+\xi_2)(\xi_2+\xi_3)} + \frac{\xi_3 a(\xi_3) + \xi_4 a(\xi_4)}{(\xi_2+\xi_3)(\xi_1+\xi_2)}.
\end{equation*}
Using the notation from the proof of Proposition \ref{prop:energyMultiplierBounds}, we have
\begin{equation*}
b_4^a(\xi_1,\xi_2,\xi_3,\xi_4) = \frac{q(\xi_1,\xi_2)}{\xi_2+\xi_3} - \frac{q(\xi_3,\xi_4)}{\xi_2+\xi_3},
\end{equation*}
and the size and regularity estimates follow from the size and regularity estimates of $q$. These estimates were already discussed in Section \ref{section:EnergyEstimatesMBO}.\\
\underline{Subcase 1b} $|\xi_1| \sim |\xi_3| \gg \max( |\xi_2|, |\xi_4|)$:\\
In this subcase we find for the resonance function $|\omega(\xi_1)+\ldots+\omega(\xi_4)| \sim |\xi_1|^2$, and the size and regularity estimates for an extension of $b_4^a$ follow from considering the trivial decomposition
\begin{equation}
\label{eq:trivialDecomposition}
\sum_{i=1}^4 \frac{\xi_i a(\xi_i)}{(\xi_1+\xi_2)(\xi_2+\xi_3)}.
\end{equation}
\textbf{Case 2} $|\xi_1^*| \sim |\xi_3^*| \gg |\xi_4^*|$:\\
In this case it is clear again that the resonance function is of size $|\xi_1^*|^2$, and a suitable extension is provided through \eqref{eq:trivialDecomposition}.\\
\textbf{Case 3} $|\xi_1^*| \sim |\xi_4^*|$:\\
\underline{Subcase 3a} $\max(|\xi_1+\xi_2|,|\xi_2+\xi_3|) \ll |\xi_1^*|$:\\
We compute
\begin{equation*}
\begin{split}
b_4^a(\xi_1,\xi_2,\xi_3,\xi_4) &= \frac{a(\xi_1) \xi_1 + a(\xi_2)(\xi_2)}{(\xi_1+\xi_2)(\xi_2+\xi_3)} + \frac{a(\xi_3) \xi_3 + a(\xi_4)(\xi_4)}{(\xi_1+\xi_2)(\xi_2+\xi_3)} \\
&= \frac{q(\xi_1,\xi_2) - q(\xi_3,-\xi_1-\xi_2-\xi_3)}{\xi_2+\xi_3} \\
&= \frac{q(\xi_1,\xi_2) - q(\xi_1 + (\xi_2+\xi_3),\xi_2 - (\xi_2+\xi_3))}{\xi_2+\xi_3},
\end{split}
\end{equation*}
and the claim follows from the size and regularity properties of $q$.\\
\underline{Subcase 3b} $|\xi_1+\xi_2| \ll |\xi_1^*|$ and $|\xi_2+\xi_3| \sim |\xi_1^*|$:\\
We use the decomposition
\begin{equation*}
\begin{split}
b_4^a(\xi_1,\xi_2,\xi_3,\xi_4) &= \frac{a(\xi_1) \xi_1 + a(\xi_2) \xi_2}{(\xi_1+\xi_2)(\xi_2+\xi_3)} + \frac{a(\xi_3)\xi_3 + a(\xi_4) \xi_4}{(\xi_1+\xi_2)(\xi_2+\xi_3)} \\
&= \frac{q(\xi_1,\xi_2)}{\xi_2+\xi_3} + \frac{q(\xi_3,\xi_4)}{\xi_2+\xi_3},
\end{split}
\end{equation*}
and the claim follows from the considerations of Subcase 1a. In case $|\xi_1+\xi_2| \sim |\xi_1^*|$ and $|\xi_2+\xi_3| \sim |\xi_1^*|$ we argue mutatis mutandis.\\
\underline{Subcase 3c} $|\xi_1+\xi_2| \sim |\xi_2 + \xi_3| \sim |\xi_1^*|$:\\
The claim follows again from considering the decomposition \eqref{eq:trivialDecomposition}.
\end{proof}
In the following estimates we have decreased symmetry compared to Section \ref{section:EnergyEstimatesMBO}, but we still have the same frequency interactions. With Proposition \ref{prop:StrichartzEstimatesSEQ} playing the role of Proposition \ref{prop:StrichartzEstimatesMBO}, we can argue in most of the cases as above.

We record the estimate for the boundary term, which is derived like in Proposition \ref{prop:energyBoundaryTermBoundMBO}:
\begin{proposition}
Let $0<l<1/4<s$ and $E_i^{a,\lambda}$, $i=0,1$ be like above. Then, we find the following estimate to hold:
\begin{equation*}
|E^{a,\lambda}_1(u(t))| \lesssim E_0^{a,\lambda}(u(t)) \Vert u(t) \Vert_{H^{s,l}_{\lambda}}^2
\end{equation*}
\end{proposition}
For the remainder we derive the following estimate:
\begin{proposition}
\label{prop:RemainderEstimateDNLS}
We find the following estimate to hold:
\begin{equation*}
\begin{split}
\left| \int_0^T R_6^{a,\lambda}(u) \right| &\lesssim (E_0^{a,\lambda}(u(T)) + E_0^{a,\lambda}(u(0))) \sup_{t \in [-T,T]} \Vert u(t) \Vert^4_{H^{s,l}_{\lambda}} \\ 
&+ T (\Vert u \Vert^6_{F^{s-\tilde{\varepsilon},l}_\lambda(T)} + \Vert u \Vert^8_{F^{s-\tilde{\varepsilon},l}(T)}).
\end{split}
\end{equation*}
\end{proposition}
\begin{proof}
We consider dyadic frequency ranges $|\xi_{1i}| \sim M_i$, $i=1,2,3$ and $|\xi_i| \sim L_i$ for $i=2,3,4$ with the increasing rearrangements $M_1^* \leq M_2^* \leq M_3^*$, $L_2^* \leq L_3^* \leq L_4^*$, and as in the proof of Proposition \ref{prop:remainderEnergyEstimateMBO}, we only deal with high frequencies.\\
Further, let $K_i^*$ denote the increasing rearrangement of the union of $M_i$ and $L_j$.\\
\textbf{Case 1} $M_2^* \sim M_3^* \sim K_6^*$, $L_4^* \lesssim M_2^*$;\\
\textbf{Case 2} $L_3^* \sim L_4^* \sim K_6^*$, $M_3^* \lesssim L_3^*$: In both cases the argument from the proof of Proposition \ref{prop:remainderEnergyEstimateMBO} applies because it depends only on short-time Strichartz estimates and the symbol size and regularity.\\
\textbf{Case 3} $L_4^* \sim M_3^* \sim K_6^*$; $\max( L_3^*, M_2^*) \ll K_6^*$: This case needs more care as, due to the reduced symmetry, we can not always argue like in Case 3 from the proof of Proposition \ref{prop:remainderEnergyEstimateMBO}.\\
If $L_3 \sim K_6^*$, we have an improved estimate for the symbol, namely\\ $((K_6^*)^{2(s+\varepsilon)} K_4^*)/(K_6^*)^2$. In this case the claim can be concluded by two bilinear Strichartz estimates involving the high frequencies and two pointwise bounds. This gives
\begin{equation*}
\left| \int_0^T R^{a,\lambda}_6(M_i,L_j) \right| \lesssim T (K_1^* K_2^*)^{1/2} \frac{(K_6^*)^{2(s+\varepsilon)} K_4^*}{K_6^*} \prod_{i=1}^6 \Vert u_{k_i^*} \Vert_{F_{k_i^*,\lambda}}
\end{equation*}
with a straight-forward summation over the frequency blocks.\\
Note the symmetry between $M_1$, $M_3$ and $L_2$, $L_4$. Suppose that $L_2 \sim K_6^*$.
We write out the imaginary part to find
\begin{equation*}
\begin{split}
R_6^{a,\lambda} &= C \int_{\Gamma_6^\lambda} [b^{a}_4(\xi_{1},\xi_2,\xi_3,\xi_4) - b^a_4(\xi_1,-\xi_{11},-\xi_{12},-\xi_{13})] \\
&\quad \times \xi_{1} \hat{u}(\xi_{11}) \hat{\overline{u}}(\xi_{12}) \hat{u}(\xi_{13}) \hat{\overline{u}}(\xi_{2}) \hat{u}(\xi_3) \hat{\overline{u}}(\xi_4) d\Gamma_6^\lambda, \quad \xi_1 = \xi_{11} + \xi_{12} + \xi_{13}.
\end{split}
\end{equation*}
If $L_2 \sim M_1$, the same argument from Case 3 from the proof of Proposition \ref{prop:remainderEnergyEstimateMBO} is applicable as we find a more favorable bound for the difference of the multipliers after using the mean value theorem.\\
If $L_2 \sim M_2$, this is not the case, but the second resonance function is very favourable:
\begin{equation*}
\Omega^{(2)}(\xi_{11},\xi_{12},\xi_{13},\xi_2,\xi_3,\xi_4) = -\xi_{11}^2 + \xi_{12}^2 - \xi_{13}^2 + \xi_2^2 - \xi_3^2 + \xi_4^2 \gtrsim (K_6^*)^2.
\end{equation*}
Then, another integration by parts gives
\begin{equation*}
\begin{split}
R^{a,\lambda}_6(M_i,L_j) &= [ \Im \int_{\Gamma_6^\lambda, |\xi_i| \sim M_i,L_j} \frac{b_4^{a}(\xi_{11}+\xi_{12}+\xi_{13},\xi_2,\xi_3,\xi_4)}{\Omega^{(2)}(\xi_{11},\xi_{12},\xi_{13},\xi_2,\xi_3,\xi_4)} (\xi_{11}+\xi_{12}+\xi_{13}) \\
&\quad \times (\hat{u}(\xi_{11}) \hat{\overline{u}}(\xi_{12}) \hat{u}(\xi_{13}) \hat{\overline{u}}(\xi_2) \hat{u}(\xi_3) \hat{\overline{u}}(\xi_4) d\Gamma_6^\lambda ]_0^T \\
&\quad + [ \Im \int_{\substack{ |\xi_{12}| \sim |\xi_2| \sim K_6^*, \\ M_5^*, L_3^* \ll K_6^*}}  \frac{b_4^{a}(\xi_{11}+\xi_{12}+\xi_{13},\xi_2,\xi_3,\xi_4)}{\Omega^{(2)}(\xi_{11},\xi_{12},\xi_{13},\xi_2,\xi_3,\xi_4)} (\xi_{11}+\xi_{12}+\xi_{13}) \\
&\quad \times \hat{u}(\xi_{11}) \hat{\overline{u}}(\xi_{12}) \hat{u}(\xi_{13}) i \xi_2 ( \hat{\overline{u}}(\xi_{21}) \hat{u}(\xi_{22}) \hat{\overline{u}}(\xi_{23})) \hat{u}(\xi_3) \hat{\overline{u}}(\xi_4) d\Gamma_8^\lambda  + \ldots ] \\
&=: E_2^{a,\lambda}(u(T)) - E_2^{a,\lambda}(u(0)) + R_8^{a,\lambda}(u),
\end{split}
\end{equation*}
where we did not record the terms coming up in case the derivative hits another factor than $\hat{\overline{u}}(\xi_2)$. This case we have singled out as $|\xi_2| \sim K_6^*$ so the factor $\xi_2$ presumably gives the main contribution. From the estimates we shall see that further terms are lower order indeed.

Since $\Omega^{(2)}$ is bounded from below, we still have the necessary regularity to argue as in the proof of Proposition \ref{prop:remainderEnergyEstimateMBO}. Further, we have the size estimate
\begin{equation*}
\left| \frac{b_4^{a}(\xi_{11}+\xi_{12}+\xi_{13},\xi_2,\xi_3,\xi_4)}{\Omega^{(2)}(\xi_{11},\xi_{12},\xi_{13},\xi_2,\xi_3,\xi_4)} (\xi_{11}+\xi_{12}+\xi_{13}) \right| \lesssim \frac{a(K_6^*)}{(K_6^*)^2}.
\end{equation*}
Hence, estimating
\begin{equation*}
\begin{split}
\left| \int_{\lambda \T} P_{k_{11}} u P_{k_{12}} \overline{u} P_{k_{13}} u P_{k_2} \overline{u} P_{k_3} u P_{k_4} u dx \right| &\lesssim \Vert P_{k_6^*} u \Vert_{L^2_x(\lambda \T)} \Vert P_{k_5^*} u \Vert_{L^2_x(\lambda \T)} \\
&\quad \quad \times \prod_{i=1}^4 \Vert P_{k_i^*} u \Vert_{L^\infty_{x}(\lambda \T)}
\end{split}
\end{equation*}
gives together with the pointwise bound, after carrying out the sum over $K_i$,
\begin{equation*}
|E_2^{a,\lambda}(u(t))| \lesssim E_0^{a,\lambda}(u(t)) \Vert u(t) \Vert_{H^{s,l}_{\lambda}}^4.
\end{equation*}

Let $N_1^* \leq \ldots \leq N_8^*$ denote the increasing rearrangement of the dyadic sizes of the occurring frequencies. We shall estimate the expression like above according to the separation of the involved frequencies.\\
\textbf{Case 1} $N_6^* \ll N_7^* \sim N_8^*$, $N_5^* \ll N_6^* \sim N_8^*$: In both cases we apply two bilinear Strichartz estimates. Note that the time localization amounts to a factor of $T N_8^*$, and the two bilinear Strichartz estimates yield a gain of $(N_8^*)^{-1}$, the four pointwise bounds give a contribution $\prod_{i=1}^4 (N_i^*)^{1/2}$.\\
\underline{Subcase 1a}: $N_8^* \sim K_6^*$. Summing the derivatives $\xi_{11}+\xi_{12}+\xi_{13}$, $\xi_2$ with the same argument as in the proof of Proposition \ref{prop:remainderEnergyEstimateMBO} yields a contribution of $(N_8^*)^2$.
Further, the multiplier is estimated by $(N_8^*)^{2(s+\varepsilon)}/(N_8^*)^3$. This gives
\begin{equation*}
\left| \int_0^T R_8^{a,\lambda}(N_1^*,\ldots,N_8^*) \right| \lesssim T \frac{(N_8^*)^{2(s+\varepsilon)}}{N_8^*} \prod_{i=1}^4 (N_i^*)^{1/2} \prod_{i=1}^8 \Vert u \Vert_{F_{n_i^*,\lambda}} ,
\end{equation*}
which is actually summable for $s>1/6$.

Next, suppose that $N_8^* \gg K_6^*$. This implies $N_6^* \sim K_6^*$ or $K_6^* \sim N_5^* \ll N_6^*$. In the first scenario, the above estimate yields the following bound
\begin{equation*}
\begin{split}
\left| \int_0^T R_8^{a,\lambda}(N_1^*,\ldots,N_8^*) \right| &\lesssim T N_8^* (N_8^*)^{-1} \frac{(N_6^*)^{2(s+\varepsilon)}}{N_6^*} \frac{(N_6^*)^2}{(N_6^*)^2} \\
&\quad \quad \times \prod_{i=1}^4 (N_i^*)^{1/2} \prod_{i=1}^8 \Vert u \Vert_{F_{n_i^*,\lambda}}
\end{split}
\end{equation*}
and in the second case
\begin{equation*}
\left| \int_0^T R_8^{a,\lambda}(N_1^*,\ldots,N_8^*) \right| \lesssim T \frac{(N_5^*)^{2(s+\varepsilon)}}{N_5^*} \prod_{i=1}^4 (N_i^*)^{1/2} \prod_{i=1}^8 \Vert u \Vert_{F_{n_i^*,\lambda}}.
\end{equation*}
In both cases summing over dyadic frequencies yields
\begin{equation*}
\left| \int_0^T R_{8}^{a,\lambda} \right| \lesssim T \Vert u \Vert^8_{F_{\lambda}^{s-\tilde{\varepsilon},l}(T)}
\end{equation*}
for $s>1/6$.\\
\textbf{Case 2}: $N_4^* \ll N_5^* \sim \ldots \sim N_8^*$, $N_3^* \ll N_4^* \sim \ldots \sim N_8^*$. From the constraint on the initial frequencies $N_3^* \sim N_8^*$ is ruled out and it has to hold $N_8^* \sim K_6^*$.

Either way, one applies one bilinear Strichartz estimate on $u_{n_4^*} u_{n_8^*}$ (in the first subcase this is clear, in the latter, since there is an odd number of high frequencies, one pair is amenable to one bilinear Strichartz estimate) and three $L^6_{t,x}$-Strichartz estimates on the remaining high frequencies $u_{n_5^*}$, $u_{n_6^*}$, $u_{n_7^*}$; the other frequencies are estimated by pointwise bounds. Here, we do not have to take into account complex conjugation because we argue by Proposition \ref{prop:StrichartzEstimatesSEQ}.

This gives
\begin{equation*}
\begin{split}
\left| \int_0^T R_8^{s,a,\lambda}(N_1^*,\ldots,N_8^*) \right| &\lesssim T N_8^* (N_8^*)^{-1/2} \frac{(N_8^*)^{2(s+\varepsilon)}}{(N_8^*)^3} (N_8^*)^2 \\
&\quad \times \prod_{i=1}^3 (N_i^*)^{1/2} \prod_{i=1}^8 \Vert u_{n_i^*} \Vert_{F_{n_i^*,\lambda}},
\end{split}
\end{equation*}
which, after summation, yields the estimate
\begin{equation*}
\left| \int_0^T R_8^{s,a,\lambda} \right| \lesssim T \Vert u \Vert^8_{F^{s-\varepsilon}_\lambda(T)}
\end{equation*}
for $s>1/6$. The proof is complete.
\end{proof}
With the bound for the remainder terms and the boundary terms at disposal, the energy estimate is carried out like for the modified Benjamin-Ono equation.\\
The concluding arguments from Section \ref{section:ProofMBO} adapt mutatis mutandis.
\section*{Acknowledgements}
I would like to thank Professor Sebastian Herr for suggesting to work on quasilinear dispersive equations and helpful discussions during early stages of this work. Moreover, I am much obliged to the anonymous referee for a careful reading of an earlier manuscript, which gave rise to many improvements.

Financial support by the German Research Foundation (IRTG 2235) is gratefully acknowledged.

\end{document}